\documentclass[11pt, oneside]{amsart}

\usepackage{amsmath,ifthen, amsfonts, amssymb, 
srcltx, % COMMENT OUT
   amsopn,tikz, tkz-euclide, url}
\usetikzlibrary{decorations.markings,arrows, cd}
\usepackage{tikz-cd}
 \usepackage{dutchcal}
\usepackage{xfrac} 
\usepackage{graphicx, ifsym}
\usepackage{fullpage}

\newcommand{\showcomments}{yes}
\renewcommand{\showcomments}{no}

\newsavebox{\commentbox}
\newenvironment{com}%
{\ifthenelse{\equal{\showcomments}{yes}}%
{\footnotemark
        \begin{lrbox}{\commentbox}
        \begin{minipage}[t]{1.25in}\raggedright\sffamily\tiny
        \footnotemark[\arabic{footnote}]}
{\begin{lrbox}{\commentbox}}}%
{\ifthenelse{\equal{\showcomments}{yes}}%
{\end{minipage}\end{lrbox}\marginpar{\usebox{\commentbox}}}
{\end{lrbox}}}

\usepackage{amsmath,amsfonts, amssymb, graphicx, tikz, bbding,ifthen,srcltx,amsopn, url}
\usepackage{bm}
\usetikzlibrary{decorations.markings,arrows, intersections}
\usetikzlibrary{decorations.pathreplacing}

\newcommand{\Z}{\mathbb Z}

\DeclareMathOperator{\Art}{Art}

\DeclareMathOperator{\rank}{rk}

\newcommand{\euler}{\chi}

\theoremstyle{definition}
\newtheorem{thm}{Theorem}[section]
\newtheorem{lem}[thm]{Lemma}
\newtheorem{prop}[thm]{Proposition}

\newtheorem{exa}[thm]{Example}

\newtheorem*{question}{Question}
\newtheorem{remark}[thm]{Remark}
\newtheorem{defn}[thm]{Definition}
\newtheorem{cor}[thm]{Corollary}
\newtheorem{obs}[thm]{Observation}
\newtheorem{desc}[thm]{Description}
\newtheorem{thmx}{Theorem}
\author{Kasia Jankiewicz}
\address{Department of Mathematics, University of Chicago, Chicago, Illinois, 60637}
\email{kasia@math.uchicago.edu}
\title{Residual finiteness of certain $2$-dimensional Artin groups}

\begin{document}
\begin{com}
{\bf \normalsize COMMENTS\\}
ARE\\
SHOWING!\\
\end{com}

\begin{abstract} 
We show that many $2$-dimensional Artin groups are residually finite. 
This includes $3$-generator Artin groups with labels $\geq 4$ except for $(2m+1, 4,4)$ for any $m\geq 2$. 
As a first step towards residual finiteness we show that 
these Artin groups, and many more, split as free products with amalgamation 
or HNN extensions of finite rank free groups. Among others, this holds for all large type Artin groups with defining graph admitting an orientation, where each simple cycle is directed.
\end{abstract}


\maketitle

A group $G$ is \emph{residually finite} if for every $g\in G-\{1\}$ 
there exists a finite quotient $\phi:G\to \bar G$ such that
$\phi(g)\neq 1$. 
The main goal of this paper is to extend the list of Artin groups known to be residually finite. 
Let $\Gamma$ be a simple graph 
where each pair of vertices $a,b$ in $\Gamma$ is labelled by an integer $M_{ab}\geq 2$.
The associated \emph{Artin group} 
\[\Art_{\Gamma} = \langle a\in V(\Gamma)\mid (a,b)_{M_{ab}} = (b,a)_{M_{ab}} \text{ for }a,b\text{ joined by an edge}\rangle.\]
By $(a,b)_{M_{ab} }$ we denote the alternating word $abab\dots$ of length $M_{ab}$. 
The Artin group on two generators with the label $M$ will be denoted by $\Art_M$, 
and the Artin group with three generators and labels $M,N,P$ will be denoted by $\Art_{MNP}$.

\begin{thmx}\label{thm:main}
If $M,N,P\geq 4$ and $(M,N,P)\neq (2m+1, 4,4)$ (for any permutation), then the Artin group $A_{MNP}$ is residually finite. 
\end{thmx}

None of the groups in Theorem~\ref{thm:main} with $M,N,P<\infty$ were previously known to be residually finite. % The theorem is satsified for all $M,N,P\geq 3$, except $(M,N,P) = (3,3,2k+1)$ where $k\geq 2$. 
%We believe that groups $\Art_{33(2k+1)}$ where $k\geq 2$ are also residually finite, 
%but our methods do not work in these cases. 
We also obtain residual finiteness of many more $2$-dimensional Artin groups. 
For precise statements see Section~\ref{sec:rf of other}.
In a subsequent work \cite{JankiewiczArtinSplittings} we prove the residual finiteness of Artin groups $\Art_{2MN}$ where $M,N\geq 4$ and at least one of them is even.

Our proof of Theorem~\ref{thm:main} %and ~\ref{thm:main3} 
relies on a splitting of these Artin group 
as a free product with amalgamation or HNN extension of finite rank free groups. 
The existence of such splitting depends on the combinatorics of the defining graph. 
Recall an Artin group $\Art_\Gamma$ with the defining graph $\Gamma$
 has \emph{large type} if all labels in $\Gamma$ are at least $3$. The quotient of an Artin group, obtained by adding the relation $a^2=1$ for every $a\in V(\Gamma)$ is a Coxeter group. 
We say $\Art_{\Gamma}$ is \emph{spherical} if the corresponding Coxeter quotient is finite, 
and $\Art_{\Gamma}$ is \emph{$2$-dimensional} if no triple of generators generates a spherical Artin group. 
In particular, every large type Artin group is $2$-dimensional. 
For the definition of \emph{admissible} partial orientation of $\Gamma$, see Definition~\ref{defn:admissible}. 
We prove the following.

\begin{thmx}\label{thm:main2}
If $\Gamma$ admits an admissible partial orientation, then $\Art_{\Gamma}$ splits as a free product with amalgamation or an HNN-extension of finite rank free groups. \end{thmx}

The above theorem includes all large type Artin groups whose defining graph $\Gamma$ admits an orientation 
such that each cycle is directed. 

All linear groups are residually finite by a classical result by Mal'cev \cite{Malcev40}. 
Among Artin group very few classes are known to be residually finite, and even fewer linear. 
It was once a major open question whether braid groups are linear and 
it was proved independently by Krammer~\cite{Krammer2002} and Bigelow~\cite{Bigelow2001}. 
Later, the linearity was extended to all spherical Artin groups by Cohen-Wales~\cite{CohenWales2002}, 
and independently by Digne \cite{Digne2003}. 
The right-angled Artin groups are also well known to be linear.  
Since linearity is inherited by subgroups, any virtually special Artin group is linear. 
Artin groups whose defining graphs are forests are the fundamental groups of graph manifolds with boundary
by the work of Brunner~\cite{Brunner92} and Hermiller-Meier \cite{HermillerMeier99}, 
and so they are virtually special by the work of Liu~\cite{LiuGraphManifolds} 
and Przytycki-Wise~\cite{PrzytyckiWiseGraphManifolds}. 
Artin groups in certain classes (including $2$-dimensional, $3$-generators) 
are not cocompactly cubulated even virtually, unless they are sufficiently similar to RAAGs 
by Huang-Jankiewicz-Przytycki~\cite{HuangJankiewiczPrzytycki16} and 
independently by Haettel~\cite{HaettelArtin}. 
In particular, if $M,N,P$ are finite, none of the groups in Theorem~\ref{thm:main}, is virtually cocompactly cubulated. 
Haettel has a conjectural classification of all virtually cocompactly cubulated Artin groups~\cite{HaettelArtin}. 
Haettel also showed that some triangle-free Artin group 
act properly but not cocompactly on locally finite, finite dimensional 
CAT(0) cube complexes \cite{HaettelTriangleFreeArtin}. We note that if one of the exponents $M, N, P$ is infinite, then the residual finiteness of $\Art_{MNP}$ is well-known.

The list of other known families of residually finite Artin groups is short. An Artin group is of \emph{FC type} if every clique (i.e.\ a complete induced subgraph) in $\Gamma$ is the defining graph of a spherical Artin group. 
(Blasco-Garcia)-(Martinez-Perez)-Paris showed that FC type Artin groups with all labels even are residually finite~\cite{BlascoGarciaMartinezPerezParis19}. (Blasco-Garcia)-Juhasz-Paris showed in~\cite{BlascoGarciaJuhaszParis18} 
the residual finiteness of Artin groups with defining graph $\Gamma$ 
where the vertices of $\Gamma$ admit a partition $\mathcal P$ such that 
\begin{itemize}
\item for each  $X\in \mathcal P$ the Artin group $A_X$ is residually finite,
\item for each distinct $X,Y\in \mathcal P$ there is at most one edge in $\Gamma$
joining a vertex of $X$ with a vertex of $Y$, and
\item the graph $\Gamma/\mathcal P$ is either a forest, or a triangle free graph with even labels. The graph $\Gamma/\mathcal P$ is defined as follows. The vertices of $\Gamma/\mathcal P$ are $\mathcal P$, and an edge with label $M$ joins sets $X,Y \in \mathcal P$ if there exist $a\in X, b\in Y$ such that $M_{ab} = M$.
\end{itemize}
In \cite{JankiewiczArtinSplittings} the author proves the residual finiteness of Artin groups $\Art_{2MN}$ where $M,N\geq 4$ and at least one of them is even.

The residual finiteness of $3$-generator affine Artin groups (i.e.\ corresponding to affine Coxeter groups), 
i.e.\ $\Art_{244}$, $\Art_{236}$, $\Art_{333}$ follows from the work of Squier~\cite{Squier87}. 
Squier proved that $\Art_{244}$ splits as an HNN extension of $F_2$ by an automorphism of an index two subgroup, and both $\Art_{236}$ and $\Art_{333}$ split as $F_3*_{F_7}F_4$ 
where $F_7$ is normal and of finite index in each of the factors. 
We give a geometric proof of the Squier's splitting of $\Art_{333}$ in Example~\ref{exa:Art333}.
The subgroup $F_7$ has index three and two respectively in the factors $F_3$ and $F_4$ in the splitting of $\Art_{333}$.
This yields a short exact sequence of groups 
\[1\to F_7\to \Art_{333}\to \Z/3*\Z/2\to 1.\]
In particular $\Art_{333}$ is free-by-(virtually free), and therefore virtually free-by-free. 
%In particular, $\Art_{333}$ is virtually a split extension of a finite rank free group by a free group. 
Since every split extension of a finitely generated residually finite group by a residually finite group 
is residually finite \cite{Malcev83}, 
we can conclude that $\Art_{333}$ is residually finite. 
Similar arguments yield residual finiteness of $\Art_{244}$ and $\Art_{236}$. 
The residual finiteness of $\Art_{333}$ and $\Art_{244}$ also follows from the fact that 
they are commensurable with the quotients of spherical Artin groups modulo their centers, 
respectively $\Art_{233}/Z$ and $\Art_{234}/Z$ \cite{CharneyCrisp2005}.

Theorem~\ref{thm:main2} provides a splitting of $\Art_{\Gamma}$ as a graph of groups with free vertex groups.
In general, the existence of such a splitting does not guarantee residual finiteness. 
In order to prove Theorem~\ref{thm:main} we carefully analyze the splitting and 
use a criterion for residual finiteness of certain amalgams of special form. 
See Theorem~\ref{thm:conditions for rf} and Theorem~\ref{thm:conditions for rf hnn}.
The following question is open in general.
\begin{question}
Let $A,B,C$ be finite rank free groups. When is the group $A*_CB$  (or $A*_B$) residually finite?
\end{question}
One instance where $G = A*_CB$ (or $A*_B$) is residually finite is when $C$ is malnormal in $A,B$.
By the combination theorem of Bestvina-Feighn \cite{BestvinaFeighn92}, 
if $A,B$ are hyperbolic, and $C$ is quasi-convex in both $A$ and $B$ and malnormal in at least one of $A,B$, 
then $G=A*_CB$ is hyperbolic. Wise showed that in such a case, $G$ is residually finite \cite{WisePolygons}, 
and later Hsu-Wise proved that $G$ is in fact virtually special \cite{HsuWiseCubulatingMalnormal}.
Another class of examples of residually finite amalgams are 
doubles of free groups along a finite index subgroup. 
These groups are virtually direct products of two finite rank free groups \cite{BDGM2001}.

On the other hand there are examples of amalgamated products of free groups that are not residually finite. 
Bhattacharjee constructed a first example which is an amalgam of two free groups along a common subgroup 
of finite index in each of the factors~\cite{Bhattacharjee94}. 
More examples are lattices in the automorphism group of a product of two trees, 
which split as twisted doubles of free groups along a finite index subgroup, 
and they were constructed by \cite{Wise96Thesis} and \cite{BurgerMozes97}.
 The Burger-Mozes examples are not only non residually finite, but virtually simple. 

The paper is organized as follows.
In Section~\ref{sec:prem} we fix notation and recall some geometric group theory tools that we use later. 
In Section~\ref{sec:rf} we recall some facts about residual finiteness 
and prove our criterion for residual finiteness of twisted doubles of free group (Theorem~\ref{thm:conditions for rf}) and of HNN extensions of free groups (Theorem~\ref{thm:conditions for rf hnn}).
In Section~\ref{sec:Artin} we recall the definition of Artin groups, 
and describe their non-standard presentations due to Brady-McCammond \cite{BradyMcCammond2000}. 
In Section~\ref{sec:Decomposition} we carefully study the presentation complex from the previous section 
and prove Theorem~\ref{thm:main2} (as Theorem~\ref{thm:decomposition}). 
Finally, in Section~\ref{sec:rf of 3 gen} we prove Theorem~\ref{thm:main} 
(as Corollary~\ref{cor:at least one even} and Corollary~\ref{cor:all odd}). 
A proof in the case where at least one label is even, is generalized to a broader family of Artin groups in Section~\ref{sec:rf of other}.
\subsection*{Acknowledgements} The author would like to thank Piotr Przytycki and Dani Wise for helpful conversations. 
She is also very grateful to anonymous referees for their corrections and suggestions.

%%%%%%%%%%%%%%%%%%%%%%%%%%%%%%%%%%%%%%%%%%%%%%%%%%%%
%%%%%%%%%%%%%%%%%%%%%%%%%%%%%%%%%%%%%%%%%%%%%%%%%%%%
\section{Graphs}\label{sec:prem}
%%%%%%%%%%%%%%%%%%%%%%%%%%%%%%%%%%%%%%%%%%%%%%%%%%%%
%%%%%%%%%%%%%%%%%%%%%%%%%%%%%%%%%%%%%%%%%%%%%%%%%%%%
In this section we gather together some standard notions and tools that we use in later sections.

%%%%%%%%%%%%%%%%%%%%%%%%%%
\subsection{Basic definitions}
%%%%%%%%%%%%%%%%%%%%%%%%%%
A graph is a $1$-dimensional CW-complex.
All the graphs we consider are finite. 
The vertex set of a graph $X$ is denoted by $V(X)$, and its edge set is denoted by $E(X)$.
Most graphs we consider are multigraphs, i.e.\  they may have multiple edges with the same endpoints, 
and \emph{loops}, i.e.\ edges with the same both endpoints. 
We refer to graphs without loops and multiple edges with equal endpoints as \emph{simple graphs}. 

A map $\rho$ between graphs is \emph{combinatorial} if the image of each vertex is a vertex, and while restricted to an open edge with endpoints $v_1, v_2$ it is a homeomorphism onto an edge with endpoints $\rho(v_1), \rho(v_2)$.
A combinatorial map $\rho:Y\to X$ between graph $X,Y$ is a \emph{combinatorial immersion}, if for every vertex $v\in Y$ and oriented edges $e_1, e_2$ with terminal vertex $v$ 
such that $\rho(e_1) = \rho(e_2)$, we have $e_1 = e_2$.
A combinatorial immersion $\rho:Y\to X$ induces 
an injective homomorphism $\pi_1 (Y,y)\hookrightarrow \pi_1 (X,x)$ \cite[Prop 5.3]{Stallings83} 
where $x,y$ are basepoints of $X,Y$ respectively with $\rho(y) = x$. 
A different basepoint $y'$ in the same connected component of $Y$ as $y$ and such that $\rho(y') = x$ represents 
a subgroup $\pi_1(Y,y')\hookrightarrow \pi_1(X,x)$ which is conjugate to $\pi_1 (Y,y)$.

Let $I_n$ denote a graph with vertex set $\{0,1,\dots, n\}$ 
with an edge for every pair of vertices $k_1, k_2$ such that $|k_2-k_1| =1$. 
Let $C_n$ denote graph $I_{n-1}$ with an additional edge joining $n-1$ and $0$.
A \emph{path of length $n$} in a graph $X$, is a combinatorial immersion $I_n\to X$.
A \emph{cycle of length $n$} in a graph $X$, is a combinatorial immersion $C_n\to X$.
We say a path or cycle is \emph{simple}, 
if vertices $0,\dots, n-1$ are mapped to distinct vertices in $X$.
We say a path is \emph{closed}, if $0$ and $n$ are mapped to the same vertex in $X$. 
%i.e.\ the path factors through a cycle of the same length.
A \emph{segment} in $X$ is a simple path whose only vertices that are mapped to vertices of valence $>2$ in $X$ are its endpoints. We refer to vertices of valence $>2$ as \emph{branching vertices}.

Suppose $X$ has a single vertex, i.e.\ $X$ is a wedge of loops. 
Let $\rho:Y\to X$ be a combinatorial immersion. 
If we choose an orientation for each edge of $X$, 
then the map $Y\to X$ can be represented by the graph $Y$ 
with edges oriented and labelled by $E(X)$. 
Visually, we pick a distinct color for each edge of $X$ and 
represent $Y\to X$ as $Y$ with edges oriented and colored. 
%We say a cycle or a path in $Y$ is \emph{monochrome}, 
%if it is mapped onto a single loop in $Y$.

If $\Gamma$ is a simple graph, 
we can describe a path as an $n$-tuple $(a_1,a_2,\dots, a_{n})$ 
of vertices of $\Gamma$ where $\{a_i, a_i+1\}$ forms an edge for each $1\leq i <n$.
Similarly we can describe a cycle in $\Gamma$ as an $n+1$-tuple $(a_1,a_2,\dots, a_{n}, a_1)$,
 if $(a_1,a_2,\dots, a_{n}, a_1)$ is a path.

%%%%%%%%%%%%%%%%%%%%%%%%%%
\subsection{Fiber product of graphs}\label{sec:fiber}
%%%%%%%%%%%%%%%%%%%%%%%%%%
Let $\rho_i:(Y_i, y_i)\to (X, x)$ is a combinatorial immersion of based graphs for $i=1,2$. 
%and let $y_1,y_2, x$ be basepoints in $Y_1, Y_2, X$ respectively, 
%with $\rho_i(y_i) = x$ for $i=1,2$.
The intersection of subgroups $\pi_1(Y_1, y_1)$ and $\pi_1(Y_2,y_2)$ of $\pi_1 (X,x)$ can be computed 
as the fundamental group of the fiber product of based graphs, by Stallings \cite{Stallings83}.
The \emph{fiber product of $Y_1$ and $Y_2$ over $X$}
is the pullback in the category of graphs,
i.e.\ it is the graph $Y_1\otimes_X Y_2$ with the vertex set
\[
\{(v_1,v_2)\in V(Y_1)\times V(Y_2): \rho_1(v_1) = \rho_2(v_2)\}
\] 
and the edge set
\[
\{(e_1, e_2)\in E(Y_1)\times E(Y_2): \rho_1(e_1) = \rho_2(e_2)\}
\]
where $\rho_1(e_1) = \rho_2(e_2)$ is the equality of oriented edges.
The graph $Y_1\otimes_X Y_2$ often has several connected components. 
The natural combinatorial immersion $Y_1\otimes_X Y_2\to X$ 
induces the embedding $\pi_1(Y_1\otimes_X Y_2, (y_1, y_2))\to \pi_1(X,x)$. 
By \cite[Thm 5.5]{Stallings83}, $\pi_1(Y_1\otimes_X Y_2, (y_1, y_2))$ is the intersection 
of $\pi_1(Y_1, y_1)$ and $\pi_1(Y_2,y_2)$ in $\pi_1 (X,x)$.
See also \cite[Section 9]{KapovichMyasnikov2002}.

Suppose $X$ has a unique vertex $y$. 
Then $V(Y_1\otimes_X Y_2) = V(Y_1)\times V(Y_2)$.
If $\rho:Y\to X$ is a combinatorial immersion of graphs then connected components of $Y\otimes_X Y$
represent the intersections $H\cap H^g$ where $H:=\pi_1 (Y,y)<\pi_1(X,x)$ and $g\in\pi_1(X,x)$.
In particular, one of the connected components of $Y\otimes_X Y$
is a copy of $Y$ with the vertex set $\{(v,v): v\in V(Y)\}$. 
It corresponds to the intersection $H\cap H^g = H$ where $g\in H$. 
We refer to this connected component of $Y\otimes_X Y$ as \emph{trivial}.
All other subgroups of the form $H\cap H^g$ are either $\{e\}$, 
or their conjugacy classes are represented by nontrivial connected components of $Y\otimes_X Y$.

\section{Residual finiteness}\label{sec:rf}
%%%%%%%%%%%%%%%%%%%%%%%%%%%%%%%%%%%%%%%%%%%%%%%%%%%%
%%%%%%%%%%%%%%%%%%%%%%%%%%%%%%%%%%%%%%%%%%%%%%%%%%%%
A group $G$ is \emph{residually finite} if for every $g\in G-\{e\}$ there exists a finite index subgroup $G'<G$ 
such that $g\notin G'$. 
Equivalently, there exists a finite quotient $\phi:G\to \bar G$ such that $\phi(g)\neq e$. 
It is easy to see, that if $G$ has a finite index residually finite subgroup, then $G$ is residually finite.

Let $H$ be a subgroup of $G$, let $\phi:G\to \bar G$ be a (not necessarily finite) quotient 
and let $\{g_i\}_i\subseteq G-H$ be a collection of elements. 
We say $\phi$ \emph{separates} $H$ from $\{g_i\}_i$ if $\phi(g_i)\notin \phi(H)$ for all $i\in I$.
A subgroup $H<G$ is \emph{separable} if 
for every finite collection $\{g_i\}_i\subseteq G-H$, 
there exists a finite quotient $\phi:G\to \bar G$ that separates $H$ from $\{g_i\}_i$. 
Equivalently, there exists a finite index subgroup $G'<_{f.i}G$ containing $H$ such that $g_i\notin G'$ for all $i$. 
To see the equivalence of the two definitions, in one direction take $N$ to be the normal core of $G'$ in $G$ 
(i.e.\ the intersection of all conjugates of $G'$ in $G$) and set $\bar G = G/N$. 
Conversely, take $G' = \phi^{-1}(\phi(H))$.

The main goal of this section is to formulate our criterion for residual finiteness
of certain free products of amalgamation and HNN extensions,
Theorem~\ref{thm:conditions for rf} and Theorem~\ref{thm:conditions for rf hnn}.
 We use the following criterion of Wise
 for residual finiteness of graph of free groups \cite{WisePolygons}. 
A graph of groups is \emph{algebraically clean}, 
if vertex groups are free, and edge groups are free factors in both of their vertex groups. 
\begin{thm}\cite[Thm 3.4]{WisePolygons}\label{thm:wise clean} 
Let $G$ split as a finite algebraically clean graph of groups 
where all edge groups are of finite rank. 
Then $G$ is residually finite.
\end{thm}

%%%%%%%%%%%%%%%%%%%%%%%%%%
\subsection{Free factor and separability}
%%%%%%%%%%%%%%%%%%%%%%%%%%
Let $H, G$ be finite rank free groups.
A famous theorem by Marshall Hall \cite{Hall49} states 
that every finitely generated subgroup of a free group is virtually a free factor, 
i.e.\ if $H<G$ then there exists a finite index subgroup $G'<G$ 
such that $H<G'$ and $H$ is a free factor of $G'$. 
A closely related result states that free groups are \emph{subgroup separable}, 
i.e.\ every finitely generated subgroup is separable.

Let $X, Y$ be graphs with basepoint $x_0, y_0$ respectively.
%In this section, let $X$ be a bouquet  of loops with a single vertex $x_0$, and let $Y$ be a graph with the basepoint $y_0$. 
Let $\rho:(Y, y_0)\to (X, x_0)$ be a combinatorial immersion inducing the inclusion of finite rank free group $H:= \pi_1(Y,y_0)\hookrightarrow \pi_1(X,x_0)=:G$. 

\begin{defn}\label{defn:oppressive}
Let $\mathcal A_\rho\subseteq G$ consist of all $g\in G$ represented by a cycle $\gamma$ in $X$  
such that $\gamma$ is a concatenation of paths $\gamma_1\cdot \gamma_2$ where:
\begin{itemize}
\item $\gamma_1 = \rho(\mu_1)$ and $\mu_1$ is a non-trivial simple non-closed path in $Y$ going from $y_0$ to some vertex $y_1$, 
\item $\gamma_2 = \rho(\mu_2)$ and $\mu_2$ is either trivial, or is a simple non-closed path in $Y$ going from some vertex $y_2$ to $y_0$, where $y_1\neq y_2\neq y_0$.
\end{itemize}
We refer to $\mathcal A_{\rho}$ as the \emph{oppressive set for $H$ in $G$ with respect to ${\rho}$}. 
We say $\mathcal A$ is an \emph{oppressive set for $H$ in $G$}, if there exists a combinatorial immersion ${\rho}$ with $\mathcal A = \mathcal A_{\rho}$.
\end{defn}

%\begin{lem}\label{lem:oppressive set disjoint}
%An oppressive set $\mathcal A$ for $H$ in $G$ is disjoint from $H$.
%\end{lem}
%\begin{proof} 
%Suppose that there exists $g\in \mathcal A$ such that $g\in H$. 
%Then $g$ is represented by a loop $\gamma$ which can be expressed 
%as a concatenation $\gamma_1\cdot \gamma_2$ as in Definition~\ref{defn:oppressive}. 
%Since $\rho$ is a combinatorial immersion, there is a unique path $\mu_1$ 
%starting at $y_0$ such that $\rho(\mu_1) =\gamma_1$, 
%and there is a unique path $\mu_2$ ending at $y_0$ 
%such that $\rho(\mu_2) =\gamma_2$. 
%Since $g\in H$ the path $\mu_1$ must end at the same vertex as $\mu_2$ starts. 
%This is a contradiction.
%\end{proof}

In Proposition~\ref{prop:oppressive set properties} we state some properties of the set $\mathcal A_{\rho}$. In particular, we explain the connection between the separation from the set $\mathcal A_{\rho}$ and $H$ being a free factor.
%how quotients of $G$ separating $H$ from $\mathcal A$ corresponds to $H$ being a free factor.
In one of the proofs below we use the following easy lemma, due to Karrass-Solitar.
\begin{lem}[\cite{KarrassSolitar69}]\label{lem:monotonicity}
Let $H$ be a free factor in $G$. 
Then for every finite index subgroup $G'<G$ the intersection $G'\cap H$ is a free factor in $G'$.
\end{lem}

%\begin{lem}\label{lem:marshall hall}
%If $\phi:G\to \bar G$ is a finite quotient 
%that separates $H$ from an oppressive set $\mathcal A$ for $H$ in $G$, 
%then $H\cap \ker\phi$ is a free factor in $\ker \phi$. 
%\end{lem}

\begin{prop}\label{prop:oppressive set properties}
Let $\rho:(Y, y_0)\to (X, x_0)$ be a combinatorial immersion of based graphs inducing the inclusion of finite rank free group $H:= \pi_1(Y,y_o)\hookrightarrow \pi_1(X,x_o)=:G$, and let $\mathcal A_{\rho}$ be the oppressive set for $H$ in $G$ with respect to $\rho$.
\begin{enumerate}
\item $\mathcal A_{\rho}\cap H = \emptyset$.
\item $\mathcal A_\rho=\emptyset$ if and only if $\rho$ is an embedding. 
\item For any based cover $(\hat X, \hat x_0)\to (X, x_0)$ such that $\rho$ factors through a combinatorial immersion $\hat \rho:(Y, y_0)\to (\hat X, \hat x_0)$, we have $\mathcal A_{\hat \rho} = \mathcal A_{\rho} \cap \pi_1(\hat X, \hat x_0)$.
\item If $\phi:G\to \bar G$ is a quotient 
that separates $H$ from $\mathcal A_{\rho}$, 
then $H\cap \ker\phi$ is a free factor in $\ker \phi$. 
\end{enumerate}
\end{prop}
\begin{proof}
(1) Suppose that there exists $g\in \mathcal A_{\rho}\cap H$. 
Then $g$ is represented by a loop $\gamma$ which can be expressed 
as a concatenation $\gamma_1\cdot \gamma_2$ as in Definition~\ref{defn:oppressive}. 
Since $\rho$ is a combinatorial immersion, there is a unique path $\mu_1$ 
starting at $y_0$ such that $\rho(\mu_1) =\gamma_1$, 
and there is a unique path $\mu_2$ ending at $y_0$ 
such that $\rho(\mu_2) =\gamma_2$. 
Since $g\in H$ the path $\mu_1$ must end at the same vertex as $\mu_2$ starts. 
This is a contradiction.

(2) Suppose $\mathcal A_\rho$ is not empty. That means that there exist a path $\mu_1$ joining vertices $y_0$ and $y_1$ in $Y$ and a path $\mu_2$ joining vertices $y_2$ and $y_0$ where $y_2\neq y_0,y_1$ such that $\rho(\mu_1)\cdot \rho(\mu_2)$ is a closed path. That means that $\rho(y_1) = \rho(y_2)$, i.e.\ $\rho$ is not an embedding. 
Conversely, suppose that $\rho$ is not an embedding, and let $y_1, y_2\in Y$ such that $\rho(y_1) = \rho(y_2)$ . 
Then the image under $\rho$ of a simple path from $y_0$ to $y_1$ 
concatenated with the image of a simple path going $y_2$ back to $y_0$ 
lifts to a closed path in $X$. That path corresponds to an element of $\mathcal A_{\rho}$.

(3) We first prove that $\mathcal A_{\hat\rho}\subseteq \mathcal A_\rho \cap \pi_1(\hat X, \hat x_0)$. 
By definition $\mathcal A_{\hat\rho}\subseteq \pi_1(\hat X, \hat x_0)$. 
Let $g\in \mathcal A_{\hat\rho}$ be represented by a cycle $\hat\gamma_1\cdot \hat\gamma_2$ in $\hat X$ as in Definition~\ref{defn:oppressive}. 
Then $\hat\gamma_1\cdot \hat\gamma_2$ maps to a cycle $\gamma_1\cdot\gamma_2$ in $X$ which still satisfies Definition~\ref{defn:oppressive} and so, $g\in \mathcal A_{\rho}$.
Conversely,  let $g\in \mathcal A_\rho \cap \pi_1(\hat X, \hat x_0)$ be represented by a cycle $\gamma_1\cdot \gamma_2$ in $X$. Since $g\in \pi_1(\hat X, \hat x_0)$ the cycle lifts to a cycle $\hat\gamma_1\cdot\hat \gamma_2$ in $\hat X$ based at $\hat x_0$. It follows that $g\in \mathcal A_{\hat \rho}$.

(4) Since $\phi$ separates $H$ from $\mathcal A_{\rho}$, the group $G':=\phi^{-1}\left(\phi(H)\right)$ contains $H$ but does not contain any element of $\mathcal A_{\rho}$. 
Let $(\hat X, \hat x_0)\to (X,x_0)$ be a cover corresponding to $G'$. 
Since $H\subseteq G'$, the map $\rho$ factors through $\hat \rho:(Y, y_0)\to (\hat X, \hat x_0)$. 
Since $\mathcal A_{\rho}\cap G' = \emptyset$, by (3) $\mathcal A_{\hat \rho} = \emptyset$. By (2) $\hat\rho$ is an embedding. Thus $H$ is a free factor of $G'$. 
By Lemma~\ref{lem:monotonicity}, $H\cap \ker \phi$ is a free factor in $\ker\phi$.
\end{proof}

%\begin{proof}[Proof of Lemma~\ref{lem:marshall hall}]
%We first show that every finite index subgroup $\hat G$ of $G$ that contains $H$ and 
%does not contain any element of $\mathcal A$ 
%corresponds to a finite cover $\hat X \to X$ where 
%$Y$ embeds in $\hat X$ and
%$\hat X\to X$ restricted to $Y$ is equal $\rho$.
%Indeed, since $H\subset \hat G$, for every cycle of $Y$ 
%the corresponding path in $X$ must lift to a closed path in $\hat X$.
%This defines a map from $Y\to \hat X$. 
%If this is not an embedding then there must exists two distinct vertices $y_1, y_2\in Y$ identified in $\hat X$. 
%Then the image of a simple path from $y_0$ to $y_1$ 
%concatenated with the image of a simple path going $y_2$ back to $y_0$ 
%lifts to a closed path in $\hat X$. 
%This is impossible by the assumption that 
%$\hat G$ does not contain any elements of $\mathcal A$. 
%Thus $Y$ embeds in $\hat X$ such that the covering map $\hat X \to X$ 
%restricted to $Y$ is the combinatorial immersion 
%inducing the inclusion $H\hookrightarrow G$. 
%Consequently $H$ is a free factor of $\hat G$. 
%
%Let now $\phi$ be a quotient homomorphism as assumed. 
%The group $\phi^{-1}\left(\phi(H)\right)$ has finite index in $G$, 
%and it contains $H$ but does not contain any element of $\mathcal A$. 
%Thus $H$ is a free factor of $\phi^{-1}\left(\phi(H)\right)$. 
%By Lemma~\ref{lem:monotonicity}, $H\cap \ker \phi$ is a free factor in $\ker\phi$.
%\end{proof}

The following Lemma will be used to verify that certain quotients separate a subgroup from its oppressive set.
\begin{lem}\label{lem:embedded universal covers}
Let $\rho:(Y, y_0)\to (X, x_0)$ be a combinatorial immersion of graphs where $x_0$ is the unique vertex of $X$. 
Let $Y_{\bullet}, X_{\bullet}$ be $2$-complexes with the $1$-skeletons $ Y_{\bullet}^{(1)} = Y$ and $X_{\bullet}^{(1)} = X$,
and let $\rho_{\bullet}:(Y_{\bullet}, y_0)\to (X_{\bullet}, x_0)$ be a map extending $\rho$. 
Let $\phi:\pi_1(X, x_0)\to \pi_1(X_{\bullet}, x_0)$ be the natural quotient
and suppose that $\phi(\pi_1 (Y, y_0)) = (\rho_{\bullet})_*(\pi_1 (Y_{\bullet}, y_0))$.
If the lift to the universal covers $\widetilde \rho_\bullet:\widetilde{Y_{\bullet}} \to \widetilde{X_{\bullet}}$ of $\rho_{\bullet}$ is an embedding, 
then $\phi$ separates $\pi_1(Y, y_0)$ from $\mathcal A_{\rho}$.
\end{lem}

\begin{proof}
The vertex set of $\widetilde{X_{\bullet}}$ can be identified with $\pi_1X_{\bullet}$. By assumption, we can view $\widetilde{Y_{\bullet}}$ as a subcomplex of $\widetilde{X_{\bullet}}$ whose vertex set contains vertices corresponding to $\phi(\pi_1 (Y, y_0)) = (\rho_{\bullet})_*(\pi_1 (Y_{\bullet}, y_0))\subseteq \pi_1(X_{\bullet}, x_0)$. %can be identified with the subset of $\widetilde{X_{\bullet}}$ corresponding to $\pi_1Y_{\bullet}$. 
Let $p$ the base vertex of $\widetilde{X_{\bullet}}$ representing the trivial element $e\in \widetilde{X_{\bullet}}$.

% is identified with a subset of $\widetilde{Y_{\bullet}}\subseteq \widetilde{X_{\bullet}}$ 
%which is the preimage of a single vertex $\widetilde y_0$ in $Y_{\bullet}$ under $(\widetilde {Y_{\bullet}}, \widetilde y_0)\to (Y_{\bullet}, y_0)$. 

Let $g\in \mathcal A_{\rho}$ be represented by a cycle $\gamma = \gamma_1\cdot \gamma _2$ in $X$
with $\gamma_i = \rho(\mu_i)$ as in Definition~\ref{defn:oppressive}, i.e.\ $\mu_1$ is a non-trivial simple path in $Y$ starting at $y_0$ and ending at some $y_1\neq y_0$, and $\mu_2$ is either trivial or it is a simple path in $Y$ starting at some $y_2\neq y_0, y_1$ and ending at $y_0$. The path $\mu_1$ lifts to unique paths $\widetilde \mu_1$ starting at $p$ in $\widetilde {Y_\bullet}\subseteq \widetilde {X_{\bullet}}$.
Similarly, the path $\mu_2$ lifts to unique path $\widetilde \mu_2$ ending at $g.p$ in $\widetilde {Y_\bullet}\subseteq \widetilde {X_{\bullet}}$.
To prove that $\phi$ separates $\pi_1(Y, y_0)$ from $\mathcal A_{\rho}$, we need to show that $\phi(g)\notin (\rho_\bullet)_* (\pi_1(Y_\bullet, y_0))$.
%The element $\phi(g)$ can be represented in $\widetilde{X_{\bullet}}$ as a path which is a concatenation of lifts of $\gamma_1,\gamma_2$. If 
%Consider the lift $\widetilde \mu_1$ of $\mu_1$ to $\widetilde Y_\bullet$ which starts at $\widetilde y_0$ and ends at some lift $\widetilde y_1$ of $y_1$.
%Similarly let $\widetilde \mu_2$ be a lift of $\mu_2$ to $\widetilde Y_\bullet$ which ends at $\widetilde y_0$ and starts at some lift $\widetilde y_2$ of $y_2$.
%Since $\widetilde{Y_{\bullet}} \to \widetilde{X_{\bullet}}$ is an embedding $\widetilde \mu_1$ and $\widetilde y_2$ can be viewed as paths in $\widetilde{X_{\bullet}}$ which 

Suppose to the contrary, that $\phi(g) \in (\rho_\bullet)_* (\pi_1(Y_\bullet, y_0))$.
%The path $\mu_1$ lifts to unique paths $\widetilde \mu_1$ starting at $p$ in $\widetilde Y_\bullet\subseteq \widetilde X_{\bullet}$.
%Similarly, the path $\mu_2$ lifts to unique path $\widetilde \mu_2$ ending at $g.p$ in $\widetilde Y_\bullet\subseteq \widetilde X_{\bullet}$.
That means that $\widetilde \mu_2$ starts where $\widetilde \mu_1$ end, so the concatenation $\widetilde \mu_1\cdot\widetilde \mu_2$ is a path from $p$ to $g.p$.
Since $\widetilde \mu_1\cdot\widetilde \mu_2$ projects onto $\mu_1\cdot \mu_2$ in $Y\subseteq Y_{\bullet}$, we conclude that $\mu_1\cdot \mu_2$ is a cycle in $Y$, which is a contradiction.
%Thus $g\in \pi_1 Y$,  
%which is impossible by Proposition~\ref{prop:oppressive set properties}(1).
\end{proof}

%%%%%%%%%%%%%%%%%%%%%%%%%%
\subsection{Residual finiteness of a twisted double}
%%%%%%%%%%%%%%%%%%%%%%%%%%
Throughout this section $A$ is a finite rank free group, $C<A$ is a finitely generated subgroup and $\beta:C
\to C$ is an automorphism.

\begin{defn}
The \emph{double of $A$ along $C$ twisted by $\beta$}, 
denoted by $D(A,C, \beta)$ 
is a free product with amalgamation $A*_CA$ 
where $C$ is mapped to the first factor via the natural inclusion $C\hookrightarrow A$, 
and to the second factor via the natural inclusion precomposed with $\beta$.
\end{defn}

\begin{prop}\label{prop:residual finiteness of a double}
Let $\mathcal A$ be an oppressive set for $C$ in $A$. 
Suppose there exists a finite quotient $\Psi:D(A,C,\beta)\to K$ 
such that $\Psi|_{A}:A\to K$ separates $C$ from $\mathcal A$. 
Then $D(A,C,\beta)$ virtually splits as an algebraically clean graph of finite rank free groups. 
In particular, $D(A,C,\beta)$ is residually finite.
\end{prop}

\begin{proof} 
The group $D(A,C,\beta)$ acts on its Bass-Serre tree $T$ with vertex stabilizers conjugate to $A$, 
and edge stabilizers conjugate to $C$. 
The group $\ker \Psi$ acts on $T$ with a finite fundamental domain, 
since the index of $\ker \Psi$ in $D(A,C,\beta)$ is finite. 
The vertex stabilizers are conjugates of $\ker \Psi\cap A = \ker \Psi|_A$, 
and the edge stabilizers are conjugates of $\ker \Psi\cap C = \ker\Psi|_A\cap C$. 
By Proposition~\ref{prop:oppressive set properties}(4), $\ker \Psi|_A\cap C$ is a free factor in $\ker \Psi|_ A$,
i.e.\ every edge stabilizer is a free factor 
in each respective vertex stabilizers of the action of $\ker \Psi$ on $T$.
In particular, $\ker \Psi$ splits as a clean graph of free groups, 
so by Theorem~\ref{thm:wise clean} $\ker \Psi$ is residually finite. 
Since $\ker \Psi$ has finite index in $D(A, C,\beta)$ the conclusion follows.
\end{proof}

A subgroup $H$ is \emph{malnormal} in $G$, 
if for every $g\in G-H$ we have $H^g\cap H = \{1\}$, where $H^g := g^{-1}Hg$. 
More generally, a collection $\{H_1,\dots, H_n\}$ of subgroups of $G$ is \emph{malnormal} in $G$, 
if for every $1\leq i,j\leq n$ and $g\in G$, we have $H_i^g\cap H_j = \{1\}$, unless $i=j$ and $g\in H_i$.

Let $\phi:A\to \bar A$ be a quotient and let $\bar C := \phi(C)$. 
The automorphism $\beta:C\to C$ projects to an automorphism $\bar \beta:\bar C \to \bar C$
if and only if $\beta(C\cap\ker\phi) = C\cap\ker\phi$. 
When that is the case, then
$\phi$ induces a quotient $\Phi:D(A,C,\beta) \to D(\bar A, \bar C,\bar \beta)$.

\begin{thm}\label{thm:conditions for rf}
Suppose there exists a quotient $\phi:A\to \bar A$ such that 
\begin{enumerate}
\item $\bar A$ is a virtually special hyperbolic group,
\item $\bar C:=\phi(C)$ is malnormal and quasiconvex in $\bar A$,
\item $\phi$ separates $C$ from an oppressive set $\mathcal A$ of $C$ in $A$,
\item $\beta$ projects to an automorphism $\bar\beta:\bar C \to \bar C$.
\end{enumerate}
Then $D(A,C,\beta)$ virtually splits as an algebraically clean graph of finite rank free groups. 
In particular, $D(A,C,\beta)$ is residually finite.
\end{thm}

\begin{proof}
Condition (4) ensures that $\phi$ extends to the quotient $\Phi:D(A,C,\beta)\to D(\bar A, \bar C,\bar \beta)$. 
Since $\phi$ separates $C$ from $\mathcal A$, 
the set $\phi(\mathcal A) = \{\phi(a) \mid a\in \mathcal A\}\subseteq \bar A$ is disjoint from $\bar C$. 
By Bestvina-Feighn~\cite{BestvinaFeighn92} $D(\bar A, \bar C,\bar\beta)$ is hyperbolic, 
since it is a free product of  two copies of a hyperbolic group $\bar A$ amalgamated along a subgroup $\bar C$ 
which is malnormal and quasiconvex in each of the factors (see also~\cite{KharlampovichMyasnikov98}). 
Since $\bar A$ is virtually cocompactly special, 
by Hsu-Wise~\cite{HsuWiseCubulatingMalnormal} $D(\bar A, \bar C,\bar\beta)$ is cocompactly cubulated. 
Then by Haglund-Wise~\cite{HaglundWiseAmalgams} $D(\bar A, \bar C,\bar\beta)$ is virtually special 
and in particular QCERF \cite{HaglundWiseSpecial}. 
Thus, $\bar C$ is separable in $D(\bar A, \bar C,\bar\beta)$. 
There exists a finite quotient $\Psi:D(\bar A, \bar C,\bar\beta)\to K$ such that 
$\Psi|_{\bar A}$ separates $\bar C$ from $\phi(\mathcal A)$. 
Thus the composition $\Psi\circ \Phi|_{A}:A\to K$ separates $C$ from $\mathcal A$.
The quotient $\Psi\circ \Phi:D(A,C,\beta)\to K$ satisfies the assumptions 
of Proposition~\ref{prop:residual finiteness of a double}.
Hence $D(A,C,\beta)$ is residually finite.
\end{proof}

In our application of Theorem~\ref{thm:conditions for rf}, Condition (4) will be verified using the following.
\begin{obs}\label{obs:beta bar}
Let $Z$ be a finite graph and let $b:(Z,z_0)\to (Z,z_1)$ be a graph automorphism.
Then $b$ together with a choice of a path from $z_0$ to $z_1$ 
induces an automorphism $\beta:\pi_1(Z, z_0)\to \pi_1 (Z,z_0)$. 
 If $Z_{\bullet}$ is a finite $2$-complex with the $1$-skeleton $Z$ 
 such that $b$ extends to $b_{\bullet}:Z_{\bullet}\to Z_{\bullet}$, 
 then $\beta$ projects to an automorphism $\beta_{\bullet}:\pi_1(Z_{\bullet}, y_0)\to \pi_1 (Z_{\bullet}, y_0)$.
\end{obs}

%%%%%%%%%%%%%%%%%%%%%%%%%%
\subsection{Residual finiteness of an HNN extension}
%%%%%%%%%%%%%%%%%%%%%%%%%%
Let $A, B$ be finite rank free groups.  
For $i=1,2$ let $\beta_i:B\to A$ denote an injective homomorphism, and denote $B_i=\beta_i(B)$.
Let $\beta$ denote the isomorphism $\beta_2\cdot\beta_1^{-1}:B_1\to B_2$.
By $A*_{B}$  we denote the HNN extension of $A$ with respect to $\{\beta_1, \beta_2\}$, i.e.
\[
A*_{B} = \langle A, t\mid t^{-1}bt = \beta(b) \text{ for all }b\in B_1\rangle
\]
\begin{com}Is that the def of hnn extension we want? Is the abuse of the notation ok? B or C?\end{com}
Let $X$ be a bouquet of loops, with $\pi_1X$ identified with $A$,
and for $i=1,2$ let  $\rho_i:Y_i\to X$ be a combinatorial immersion inducing the inclusion $B_i\hookrightarrow A$.
Let $\mathcal A_{\rho_1},\mathcal A_{\rho_2}$ be the oppressive sets for $B_1,B_2$ with respect to $\rho_1, \rho_2$ respectively. 

\begin{prop}\label{prop:residual finiteness of hnn}
Suppose there exists a quotient $\Psi: A*_B\to K$ such that $\Psi|_{A}:A\to K$ separates $B_1$ from $\mathcal A_1$, and $B_2$ from $\mathcal A_2$. Then $A*_B$  virtually splits as an algebraically clean graph of finite rank free groups. 
In particular, $A*_B$ is residually finite.
\end{prop}

\begin{proof}
The proof is analogous as for Proposition~\ref{prop:residual finiteness of a double}.
\end{proof}

\begin{thm}\label{thm:conditions for rf hnn}
Suppose there exists $\phi:A\to \bar A$ such that
\begin{enumerate}
\item $\bar A$ is a virtually special hyperbolic group,
\item $\bar B_i := \phi(B_i)$ is quasiconvex in $\bar A$ for $i=1,2$, and the collection $\{\bar B_1, \bar B_2\}$ is malnormal in $\bar A$,
\item $\phi$ separates $B_i$ from $\mathcal A_{\rho_i}$ for $i=1,2$,
\item $\beta$ projects to an isomorphism $\bar\beta:\bar B_1\to \bar B_2$.
\end{enumerate}
Then $A*_B$  virtually splits as an algebraically clean graph of finite rank free groups. 
In particular, $A*_B$ is residually finite.
\end{thm}
\begin{proof}
The proof is analogous as the proof of Theorem~\ref{thm:conditions for rf}. It uses Proposition~\ref{prop:residual finiteness of hnn} in the place of Proposition~\ref{prop:residual finiteness of a double}.
\end{proof}

%%%%%%%%%%%%%%%%%%%%%%%%%%%%%%%%%%%%%%%%%%%%%%%%%%%%
%%%%%%%%%%%%%%%%%%%%%%%%%%%%%%%%%%%%%%%%%%%%%%%%%%%%
\section{Artin groups and their Brady-McCammond complex}\label{sec:Artin}
%%%%%%%%%%%%%%%%%%%%%%%%%%%%%%%%%%%%%%%%%%%%%%%%%%%%
%%%%%%%%%%%%%%%%%%%%%%%%%%%%%%%%%%%%%%%%%%%%%%%%%%%%
In this section we describe a complex $X_{\Gamma}$ associated to a non-standard presentation of $\Art_{\Gamma}$
that was introduced and shown to be CAT(0) for many Artin groups 
by Brady-McCammond in~\cite{BradyMcCammond2000}. 
We then describe certain subspaces of $X_{\Gamma}$ 
that will be used in Section~\ref{sec:Decomposition} 
to prove that for certain $\Gamma$ the group $\Art_{\Gamma}$ splits 
as an amalgam of finite rank free groups. 
We start with the case of $2$-generator Artin group.

%%%%%%%%%%%%%%%%%%%%%%%%%%
\subsection{Brady-McCammond presentation for a 2-generator Artin group}\label{sec:2gen}
%%%%%%%%%%%%%%%%%%%%%%%%%%
Consider an Artin group on two generators
\[\Art_M = \langle a,b \mid (a,b)_M = (b,a)_M\rangle\]
where $M<\infty$.
By adding an extra generator $x$ and setting $x=ab$ we get another presentation

\begin{itemize}
	\item if $M = 2m$:
	\[\langle a,b,x\mid x=ab, x^m = bx^{m-1}a\rangle\]
	\item if $M=2m+1$:
	\[\langle a,b,x\mid x=ab, x^ma = bx^{m}\rangle\]
\end{itemize}
See Figure~\ref{fig:presentation}.
\begin{figure}
	\includegraphics[scale=0.2]{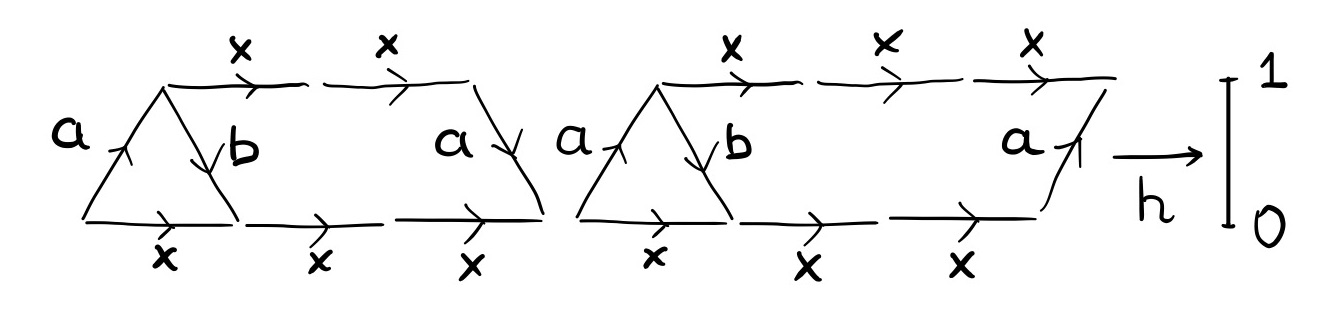}
	\caption{The $2$-cells of the presentation complex of the group presentation of 
	a $2$-generator Artin group $\Art_{5}$ (left) and $\Art_{6}$ (right).}\label{fig:presentation}
\end{figure}
Let $r_{2m}(a,b,x)$ denote the relation $x^m = bx^{m-1}a$ and 
let $r_{2m+1}(a,b,x)$ denote the relation $x^ma = bx^{m}$.
Let $X(a,b)$ be the $2$-complex corresponding to the above presentation. 
Denote by $C(a,b)$ the disjoint union of its two $2$-cells, 
and let $p:C(a,b)\to X(a,b)$ be the natural projection.
There is an embedding of $C(a,b)$ in the plane 
and a height map $h$ to the interval $[0,1]$ 
such that $h$ restricted to each edge $x$ is constant, 
as in Figure~\ref{fig:presentation}. 
We refer to these edges as \emph{horizontal}, and to the other edges as \emph{non-horizontal}.
Note that the map $h$ is not well-defined on $X(a,b)$.

%%%%%%%%%%%%%%%%%%%%%%%%%%
\subsection{Brady-McCammond presentation for a general Artin group}\label{sec:gen Artin}
%%%%%%%%%%%%%%%%%%%%%%%%%%

A \emph{partial orientation} on a simple graph $\Gamma$ is a choice of an endpoint $\iota(e)$ for some of the edges $e$ in $E(\Gamma)$. 
Visually we represent a partial orientation on a simple graph by arrows: an edge $e$ with a choice of vertex $\iota(e)$ is represented as an arrow starting at the vertex $\iota(e)$.
We say a cycle (resp. path) $\gamma$ in a simple graph $\Gamma$ with a partial orientation $\iota$ is \emph{directed}, if for every edge $e$ in the cycle $\iota(e)$ is defined, and $\iota(e) = \iota(e')$ only when $e=e'$. An \emph{orientation} on $\Gamma$ is a partial orientation where each edge is oriented.

Let $\Gamma$ be a simple graph with edges labelled by number $\geq 2$, with a partial orientation $\iota$ where $i\iota(e)$ is defined for an edge $e$ if and only if the label of $e$ is $\geq 3$.

%A \emph{orientation} on a simple graph $\Gamma$
%assigns to each edge $e\in E(\Gamma)$ a set $o(e)$ of one or two endpoints of $e$.  
%An edge with both endpoints assigned is called \emph{bioriented}.
%We say a cycle $\gamma$ in $\Gamma$ is \emph{directed},
%if for each vertex $v\in \gamma$, there is exactly one edge $e\in \gamma$ such that $v\in o(e)$.
%We say a path $\gamma$ in $\Gamma$ is \emph{directed},
%if for each vertex $v\in \gamma$, except the first one or the last one, 
%there is exactly one edge $e\in \gamma$ such that $v\in o(e)$.
%In particular in a directed cycle or path, no edge is bioriented.
%
%Let $\Gamma$ be a simple graph with edges labelled by number $\geq 2$ and with a fixed orientation $o$ such that edge $e$ is bioriented if and only if the label of $e$ equals $2$. 
%
Generalizing Section~\ref{sec:2gen} we consider the following presentation of $\Art_{\Gamma}$ with respect to the partial orientation $\iota$:
\[
\langle a\in V(\Gamma), x \in E(\Gamma)\mid x=ab,\,  r_{M_{ab}}(a,b,x) \text{ where }x =\{a,b\} \text{ and either } a=
\iota(x) \text{ or }M_{ab}=2 \rangle.
\]
The partial orientation of the edge $x = \{c,d\}$ determines whether the new generators $x$ equals $cd$ or $dc$.
If $M_{cd}=2$ we have $x=cd = dc$, which is why we do not need to specify the partial orientation.
In the case of a $3$-generators Artin group 
\[
\Art_{MNP} = \langle a,b,c\mid(a,b)_M = (b,a)_M, (b,c)_N= (c,b)_N , (c,a)_P = (a,c)_P\rangle.
\] 
with $M,N,P<\infty$, the cyclic orientation on the triangle $\Gamma$ (see Figure~\ref{fig:orientation on a triangle}) gives the presentation
\[\langle a,b,c,x,y,z\mid x=ab, y=bc, z=ca,\\r_{M}(a,b,x), r_{N}(b,c,y), r_{P}(c,a,z)\rangle.\]
\begin{figure}
\includegraphics[scale=0.2]{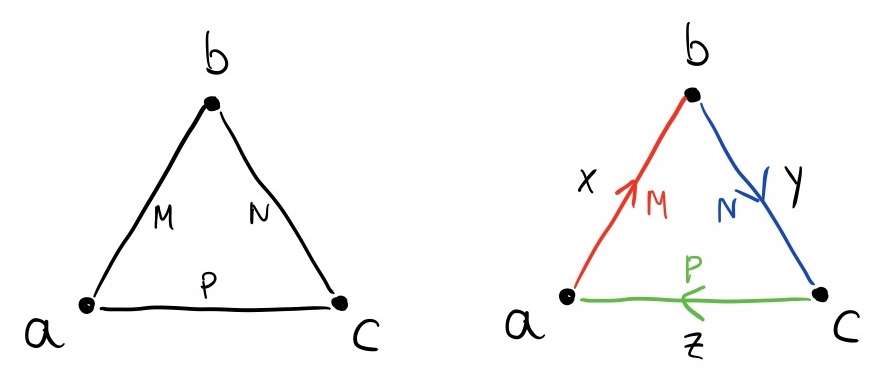}\caption{On the left: the defining graph $\Gamma$ of a triangle Artin group. On the right: the graph $\Gamma$ equipped with the cyclic orientation determining a Brady-McCammond presentation with three new generators $x,y,z$.}\label{fig:orientation on a triangle}
\end{figure}
Let $X_{\Gamma}$ be the complex obtained from the union $\bigcup_{(a,b)\in E(\Gamma)} X(a,b)$ 
by identifying the edges with the same labels. 
The fundamental group of $X_{\Gamma}$ is $\Art_{\Gamma}$. 
Brady-McCammond showed in ~\cite{BradyMcCammond2000} that when all labels are $\geq 3$,
then $X_{\Gamma}$ admits a locally CAT(0) metric provided that there exists an orientation such that
\begin{enumerate}
	\item every triangle in $\Gamma$ is directed,
	\item every $4$-cycle in $\Gamma$ contains a directed path of length at least $2$.
\end{enumerate}
Their proof in fact works in a greater generality. Using their methods one can show that for certain graphs with labels $2$ the complex $X_{\Gamma}$ admits a locally CAT(0) metric. We discuss the condition on $\Gamma$ in more detail in Section~\ref{sec:statement splitting}.

As in the $2$-generator case, let $C_{\Gamma}$ be the disjoint union of the $2$-cells of $X_{\Gamma}$. 
Again let $p:C_{\Gamma}\to X_{\Gamma}$ be the projection map.
We also define a height function $h:C_{\Gamma}\to [0,1]$ whose restriction to each $C(a,b)$ 
is the height function defined in Section~\ref{sec:2gen}.

%%%%%%%%%%%%%%%%%%%%%%%%%%%%%%%%%%%%%%%%%%%%%%%%%%%%
%%%%%%%%%%%%%%%%%%%%%%%%%%%%%%%%%%%%%%%%%%%%%%%%%%%%
\section{Splittings of Artin groups}\label{sec:Decomposition}
%%%%%%%%%%%%%%%%%%%%%%%%%%%%%%%%%%%%%%%%%%%%%%%%%%%%
%%%%%%%%%%%%%%%%%%%%%%%%%%%%%%%%%%%%%%%%%%%%%%%%%%%%

%%%%%%%%%%%%%%%%%%%%%%%%%%
\subsection{The statement of the Splitting Theorem}\label{sec:statement splitting}
%%%%%%%%%%%%%%%%%%%%%%%%%%

The main goal of Section~\ref{sec:Decomposition} is to prove Theorem~\ref{thm:decomposition}, which asserts that under certain assumption on $\Gamma$, $\Art_{\Gamma}$ splits as a free product with amalgamation $A*_CB$ or an HNN-extension $A*_B$ where $A,B,C$ are finite rank free groups.
We begin with a precise statement.

%Let $\Gamma$ be a labelled simple graph, with a partial orientation $\iota$ such that 
%an edge $e$ is oriented if and only if the label of $e$ is $\geq 3$,
%%an edge is bioriented if and only if its label is $2$,
% as in Section~\ref{sec:gen Artin}.
% 
 \begin{defn}\label{defn:misdirected} Let $\Gamma$ be a simple graph with a partial orientation $\iota$.
We say a path $\gamma$ of length $\geq 2$ in $\Gamma$ is a \emph{misdirected path} 
if the partial orientation on $\gamma$ induced by $\iota$ can be extended to an orientation such that a maximal directed subpath of $\gamma$ has length $1$.
We say an even length cycle $\gamma$ is a \emph{misdirected cycle} 
if the induced partial orientation on $\gamma$ extends to an orientation where maximal directed subpaths of $\gamma$ have length $1$.
We say a cycle $\gamma$ is an \emph{almost misdirected cycle} if $\gamma$ can be expressed as a cycle $(a_1, \dots, a_n, a_1)$ where the path $(a_1,\dots, a_n)$ is misdirected. 
\end{defn}
See Figure~\ref{fig:misdirected paths}(1) for examples of misdirected paths and a misdirected cycle.
%Note that an almost misdirected cycle might contain a directed subpath of length $2$ or $3$. 
See Figure~\ref{fig:misdirected paths}(2) for examples of almost misdirected cycles. Note that every even length misdirected cycle is almost misdirected, but not vice-versa.
\begin{figure}
\includegraphics[scale=0.3]{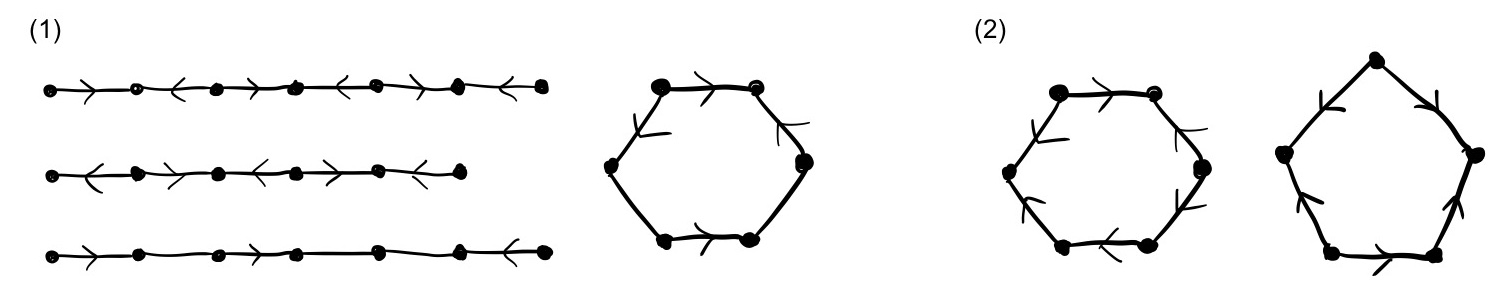}\caption{(1) Examples of misdirected paths and a misdirected cycle. (2) Almost misdirected cycles of even and odd length may contain directed subpaths of length $3$ and $2$ respectively.}\label{fig:misdirected paths}
\end{figure}
%\begin{figure}
%\includegraphics[scale=0.3]{almostmisdirectedcycles}\caption{Fully oriented almost misdirected cycles of even and odd length contain directed subpaths of length $3$ and $2$ respectively.}\label{fig:almostmisdirectedcycles}
%\end{figure}

\begin{defn}\label{defn:admissible}Let $\Gamma$ be a simple graph. Assume edges of $\Gamma$ are labelled by an integer $\geq 2$. We say that  a partial orientation $\iota$ on $\Gamma$ is \emph{admissible} if 
\begin{itemize}
\item $\iota(e)$ for an edge $e$ is defined if and only if the label of $e$ is $\geq 3$, and
\item no cycle in $\Gamma$ is almost misdirected.
\end{itemize}

\end{defn}

\begin{thm}\label{thm:decomposition}
Suppose $\Gamma$ admits an admissible partial orientation. 
%Then $\Art_{\Gamma}$ splits as a free product with amalgamation or an HNN-extension of finite rank free groups. 
%The splittings is an HNN-extension if and only if $\Gamma$ is a bipartite graph with all labels even. 
If $\Gamma$ is a bipartite graph with all labels even, then $\Art_{\Gamma}$ splits as an HNN-extension $A*_B$, where $A,B$ are finite rank free groups.
Otherwise $\Art_{\Gamma}$ splits as a free product with amalgamation $A*_CB$ where $A,B,C$ are finite rank free groups. 
Moreover, $\rank A = |E(\Gamma)|$, $\rank B =1 -|V(\Gamma)| +2|E(\Gamma)|$, and $C$ is an index $2$ subgroup of $B$, so $\rank C = 1 -2|V(\Gamma)| +4|E(\Gamma)|$. 
\end{thm}
We prove Theorem~\ref{thm:decomposition} in Section~\ref{sec:decomposition}.
The condition that $\Gamma$ has no almost misdirected cycles 
implies $\Art_{\Gamma}$ in Theorem~\ref{thm:decomposition} is $2$-dimensional (since no $3$-cycle can have an edge labelled by $2$).
Our condition also implies the other condition given by Brady-McCammond (and included in the end of Section~\ref{sec:gen Artin}) ensuring that $X_{\Gamma}$ is CAT(0). Therefore all Artin groups satisfying the assumptions of Theorem~\ref{thm:decomposition} are CAT(0) by \cite{BradyMcCammond2000}.
Since a $4$-clique does not admit an orientation where each $3$-cycle is directed, our condition also implies that the clique number of $\Gamma$ for is at most $3$.

Recall, $\Art(\Gamma)$ has \emph{large type}, if $M_{ab} \geq 3$ for all $\{a,b\}\in E(\Gamma)$.
Here are some examples of Artin groups that satisfy the assumptions of Theorem~\ref{thm:decomposition}:
\begin{itemize}
\item All large type $3$-generator Artin groups.
\item More generally, large type Artin group whose defining graph $\Gamma$ admits an orientation where each simple cycle is directed. 
This includes $\Gamma$ that is planar and each vertex have even valence (as observed in \cite{BradyMcCammond2000}).
\item Many other Artin groups with the sufficiently small ratio $\frac{\#\text{ labels }2\text{ in }\gamma}{\text{length}(\gamma)}$ in every cycle $\gamma$. In particular, this includes Artin groups with $\Gamma$ where all edges labelled by $2$ disconnect the graph and all subgraphs without edges labelled by $2$ are as above.
\end{itemize}

For the rest of this section, we assume that $\Gamma$ is a fixed connected, labelled, simple graph.
We write $X$ for the Brady-McCammond complex $X_{\Gamma}$ defined in Section~\ref{sec:gen Artin}. 
The splitting of $\Art_{\Gamma}$ comes from a decomposition of the $2$-complex $X$ 
into a union of two subspaces where each subspace 
and the intersection of them all have homotopy type of graphs. 
We will now describe these subspaces.

%%%%%%%%%%%%%%%%%%%%%%%%%%
\subsection{Horizontal graphs in $X$}\label{sec:horizontal}
%%%%%%%%%%%%%%%%%%%%%%%%%%

We distinguish the following subspaces of $X$ that are the images under $p$ of level sets of the height function $h$, as defined in Section~\ref{sec:gen Artin}.
\begin{figure}
\includegraphics[scale=0.4]{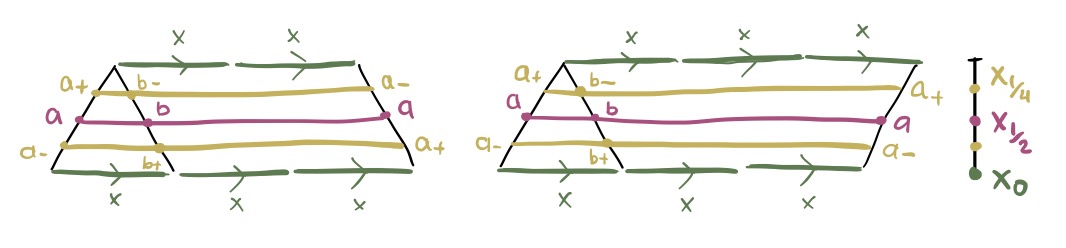}\caption{Horizontal graphs $X_0$, $X_{\sfrac 12}$ and $X_{\sfrac 14}$.}\label{fig:horizontal}
\end{figure}
\begin{itemize}
	\item[($0$)] The level set $p(h^{-1}(0))$ is denoted by $X_0$.
		The intersection of $X_0$ with every $X(a,b)$ is a single loop labelled by the generator $x=ab$.
		Thus $X_0$ is a bouquet of loops, one for each edge in $E(\Gamma)$. See Figure~\ref{fig:horizontal}.
	\item[($\sfrac 12$)] The level set $p(h^{-1}(\frac 12))$ is denoted by $X_{\sfrac{1}{2}}$. 
		We call the points of intersection of $X_{\sfrac{1}{2}}$ 
		with the non-horizontal edges the \emph{midpoints}. 
		We will abuse the notation, and $a$ will denote the midpoint of the edge labelled by $a$. 
		The intersection of $X_{\sfrac 12}$ with every $X(a,b)$ is a single cycle of length $2$ with vertices $a,b$. 
		The graph $X_{\sfrac 12}$ is a union of all these cycles of length two identified along vertices with the same label. 
		Hence $X_{\sfrac 12}$ is a copy of the graph $\Gamma$ with every edge doubled. See Figure~\ref{fig:horizontal}.
	\item[($\sfrac 14$)] The union of the level set $p(h^{-1}(\frac 1 4)\cup h^{-1}(\frac 3 4))$ 
		is denoted by $X_{\sfrac{1}{4}}$. 
		We call the points of intersection of $X_{\sfrac{1}{4}}$ 
		with the non-horizontal edges the \emph{quarterpoints}, 
		and denote them by $a_+, a_-, b_+, b_-$ where the vertices $a_-,a, a_+$ are ordered 
		with respect with the orientation of the edge $a$. 
		Similarly, $b_-,b, b_+$ are ordered with respect with the orientation of the edge $b$.
		See Figure~\ref{fig:horizontal}.
		
		If $M_{ab}$ is odd, the intersection of $X_{\sfrac{1}{4}}$ with $X(a,b)$ 
		is a single cycle of length $4$. 
		If $M_{ab}$ is even, the intersection of $X_{\sfrac{1}{4}}$ with $X(a,b)$ 
		is a disjoint union of two cycles, each of length $2$. 
		We describe $X_{\sfrac 14}$ in more detail in Section~\ref{sec:X14}.
\end{itemize}

Let us emphasize that $X_{\sfrac{1}{2}}$ is never a simple graph;
it always has double edges. Similarly $X_{\sfrac{1}{4}}$ does not need to be simple. \\

%%%%%%%%%%%%%%%%%%%%%%%%%%
\subsection{Horizontal tubular neighborhoods in $X$}\label{sec:horizontal tubular}
%%%%%%%%%%%%%%%%%%%%%%%%%%
Fix $0<\epsilon<1/4$. We now define tubular neighborhoods $N_0, N_{\sfrac 12}, N_{\sfrac 14}\subseteq X$ of graphs $X_0, X_{\sfrac 12}, X_{\sfrac 14}$.
\begin{itemize}
	\item[($0$)] Let $N_0$ be an open neighborhood of $X_0$ 
	of the form $p\left(h^{-1}([0,1/2-\epsilon)\cup(1/2+\epsilon,1])\right)$. 
	Note that $N_0$ deformation retracts onto $X_0$ with the property that 
	the intersection of $N_0$ with the $1$-skeleton of $X$ 
	is contained in the $1$-skeleton of $X$ at all times.
	
	\item[($\sfrac 12$)] Similarly, let $N_{\sfrac{1}{2}}$ be an open neighborhood of $X_{\sfrac{1}{2}}$ 
	of the form $p(h^{-1}((\epsilon, 1-\epsilon)))$. 
	Again, $N_{\sfrac{1}{2}}$ deformation retracts onto $X_{\sfrac{1}{2}}$ 
	such that $N_{\sfrac{1}{2}}\cap X^{(1)}$ is contained in $X^{(1)}$ at all times. 
	
	\item[($\sfrac 14$)] The intersection $N_0\cap N_{\sfrac{1}{2}}$, 
	which we denote by $N_{\sfrac14}$, restricted to $X(a,b)$
	 is equal to $p\left(h^{-1}((\epsilon, 1/2-\epsilon)\cup (1/2+\epsilon, 1-\epsilon))\right)$. 
	 Consequently, $N_{\sfrac14}$ deformation retracts onto $X_{\sfrac{1}{4}}$ such that
	 $N_{\sfrac 14}\cap X^{(1)}$ is contained in $X^{(1)}$ at all times. 
\end{itemize}
We also have $N_0\cup N_{\sfrac{1}{2}} = X$ 
because $[0,1/2-\epsilon)\cup(1/2+\epsilon,1]\cup(\epsilon, 1-\epsilon) = [0,1]$.

%%%%%%%%%%%%%%%%%%%%%%%%%%
\subsection{Splitting}
%%%%%%%%%%%%%%%%%%%%%%%%%%
Let $A = \pi_1 X_0 = \pi_1 N_0$, $B=\pi_1 X_{\sfrac{1}{2}} = \pi_1 N_{\sfrac{1}{2}}$ 
and if $X_{\sfrac 14}$ is connected, let 
$C =\pi_1 X_{\sfrac{1}{4}} = \pi_1 N_{\sfrac{1}{4}}$.
The group $A,B,C$ are all the fundamental groups of finite graphs, 
so they are finite rank free groups. 
The composition $X_{\sfrac{1}{4}}\hookrightarrow N_{0} \to X_0$ of the inclusion $X_{\sfrac{1}{4}}\hookrightarrow N_{0}$ with the retraction $ N_{0} \to X_0$ induces a group homomorphism $C\to A$. 
Similarly, the composition $X_{\sfrac{1}{4}}\hookrightarrow N_{\sfrac{1}{2}} \to X_{\sfrac{1}{2}}$ induces a group homomorphism $C\to B$.

When $X_{\sfrac 1 4}$ is connected, then so is $N_{\sfrac 1 4}$.
Since $N_0\cup N_{\sfrac{1}{2}} = X$ and $N_{\sfrac 1 4} =N_0\cap N_{\sfrac{1}{2}}$, by the Seifert-Van Kampen theorem we get the following.
\begin{lem}\label{lem:pushout}
If $X_{\sfrac 1 4}$ is connected and maps $C\to A$ and $C\to B$ are injective, then $\Art_\Gamma =  A*_C B$. 
\end{lem}

Analogously, we have the following.
\begin{lem}\label{lem:hnn} 
Suppose $X_{\sfrac 1 4}$ has two connected components and
$X_{\sfrac 14} \to X_{\sfrac 12}$ restricted to each connected component is a combinatorial bijection. 
If $X_{\sfrac 14} \to X_0$ restricted to each connected component is $\pi_1$-injective, 
then $\Art_\Gamma = A*_B$, where the two copies of $B$ in $A$ are induced by the two restrictions of $X_{\sfrac 14} \to X_0$ to a connected component.
\end{lem}
\begin{proof} Since $X_{\sfrac 1 4}$ has two connected components, 
$X$ is a graph of spaces with one vertex and one loop, 
where the vertex space is $X_0$ and the edge space is $X_{\sfrac 12}$ 
with two maps to $X_0$ coming from the two restrictions of $X_{\sfrac 14} \to X_0$ to a connected component. 
Since $X_{\sfrac 14} \to X_0$ restricted to each connected component is $\pi_1$-injective, 
we get the claimed HNN-extension.
\end{proof}

%%%%%%%%%%%%%%%%%%%%%%%%%%
\subsection{The graph $X_{\sfrac{1}{4}}$}\label{sec:X14}
%%%%%%%%%%%%%%%%%%%%%%%%%%
Let us first analyze $X(a,b)_{\sfrac{1}{4}}:=X_{\sfrac{1}{4}}\cap X(a,b)$. 
It has four vertices labelled by $a_+, a_-, b_+, b_-$, and four edges.
If $M_{ab}$ is even, then $X(a,b)_{\sfrac{1}{4}}$ has two edges between $a_+, b_-$ and two edges between $a_-, b_+$. 
If $M_{ab}$ is odd, then $X(a,b)_{\sfrac{1}{4}}$ is a $4$-cycle on vertices $a_+, b_-, a_-, b_+$. 
We will think of the set of edges of $X(a,b)_{\sfrac{1}{4}}$ as a disjoint union $E(a,b)'\sqcup E(a,b)''$ 
where \begin{itemize}
\item The set $E(a,b)'$ is equal $\left\{\{a_+, b_-\}, \{a_-, b_+\}\right\}$. 
Those edges correspond to the segments contained in the $2$-cell with the boundary $abx^{-1}$ in the presentation complex 
(see Figure~\ref{fig:horizontal}).
\item The set $E(a,b)''$ is equal $\left\{\{a_+, b_-\}, \{a_-, b_+\}\right\}$ or $\left\{\{a_+, b_+\}, \{a_-, b_-\}\right\}$, depending on the parity of $M_{ab}$. 
Those edges correspond to the segments contained in the $2$-cell $r_M(a,b,x)$ in the presentation complex 
(see Figure~\ref{fig:horizontal}).
\end{itemize}
This gives us the following description of $X_{\sfrac{1}{4}}$ for general $\Art_{\Gamma}$. 

\begin{desc}\label{lem:structure of X14}
The inclusion $X_{\sfrac{1}{4}}\hookrightarrow N_{\sfrac{1}{2}}$ composed 
with the deformation retraction $N_{\sfrac{1}{2}}\to X_{\sfrac{1}{2}}$ is a covering map $X_{\sfrac{1}{4}}\to X_{\sfrac{1}{2}}$ of degree $2$. Consequently, the graph $X_{\sfrac{1}{4}}$ is a double cover of the graph $X_{\sfrac{1}{2}}$ and can be described in terms of $\Gamma$ as follows:
\begin{itemize}
\item The vertex set $V(X_{\sfrac{1}{4}})$ is the disjoint union $V_+\sqcup V_-$ 
where each $V_+, V_-$ is in 1-to-1 correspondence with $V(\Gamma)$. 
For each $a\in V(\Gamma)$ we denote the corresponding vertices by $a_+, a_-$ respectively.
\item The set of edges $E(X_{\sfrac{1}{4}})$ is the disjoint union $E'\sqcup E''$ 
such that each of the graphs $\Gamma' = (V_+\sqcup V_-, E')$ 
and $\Gamma'' = (V_+\sqcup V_-, E'')$ is a double cover of $\Gamma$ 
with $a_{\pm}\mapsto a$ for every $a\in V(\Gamma)$.
In particular, $\Gamma'$ and $\Gamma''$ are simple graphs.
\item For each $\{a,b\}\in E(\Gamma)$, $E'$ contains an edge $\{a_+, b_-\}$ and $\{a_-, b_+\}$. In particular, $\Gamma'$ is a bipartite double cover of $\Gamma$.
\item For each $\{a,b\}\in E(\Gamma)$ where $M_{ab}$ is even, there is an edge $\{a_+,b_-\}$ and an edge $\{a_-,b_+\}$ in $E''$. In particular, when $M_{ab}$ is even then $E(X_{\sfrac 1 4})$ contains two copies of the edge $\{a_+, b_-\}$ and two copies of the edge $\{a_-, b_+\}$.
\item For each $\{a,b\}\in E(\Gamma)$ where $M_{ab}$ is odd, there is an edge $\{a_+, b_+\}$ and an edge $\{a_-, b_-\}$ in $E''$. 
\end{itemize}
In particular, every path $\gamma = (a_1, a_2,\dots, a_n)$ in $\Gamma$ has two lifts in $\Gamma'$: 
\begin{itemize}
\item $(a_{1+}, a_{2-},\dots, a_{n+})$ and $(a_{1-}, a_{2+},\dots, a_{n-})$, if $\gamma$ has even length, i.e.\ $n$ is odd,
\item $(a_{1+}, a_{2-}\dots, a_{n-})$ and $(a_{1-}, a_{2+},\dots, a_{n+})$, if $\gamma$ has odd length, i.e.\ $n$ is even.
\end{itemize}

Every $n$-cycle $(a_1,a_2,\dots, a_n, a_1)$ in $\Gamma$ has:
\begin{itemize} 
\item one lift $(a_{1+}, a_{2-}, \dots a_{n+}, a_{1-}, a_{2+}, \dots, a_{n-}, a_{1+})$ in $\Gamma'$ of length $2n$, if $n$ is odd,
\item two lifts $(a_{1+}, a_{2-}, \dots a_{n-}, a_{1+})$ and $(a_{1-}, a_{2+}, \dots, a_{n+}, a_{1-})$  in $\Gamma'$, each of length $n$, if $n$ is even.
\end{itemize}

\end{desc}

See Figure~\ref{fig:X1/4} for $X_{\sfrac{1}{4}}$ of a $3$-generators Artin group $\Art_{MNP}$.
\begin{figure}
\includegraphics[scale=0.3]{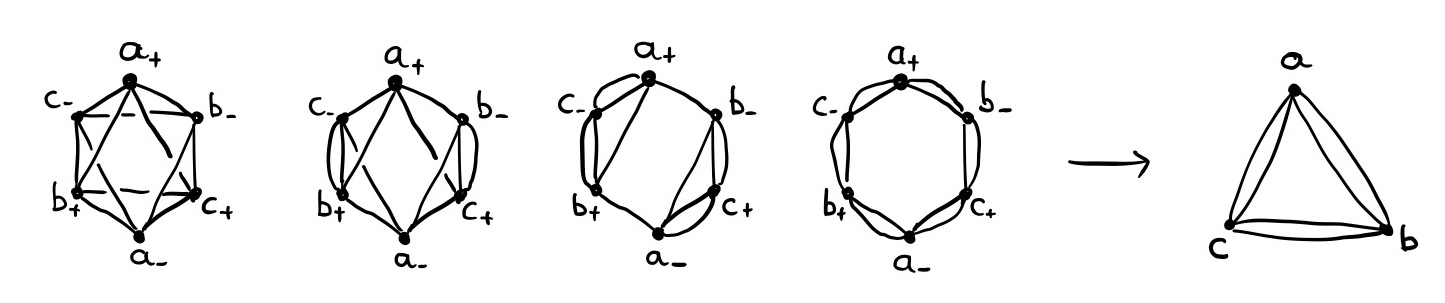}\caption{The graph $X_{\sfrac{1}{4}}$ if (1) all $M,N,P$ are odd, (2) only $N$ is even, (3) only $M$ is odd, (4) all $M,N,P$ are even. In all cases, $X_{\sfrac{1}{4}}\to X_{\sfrac{1}{2}}$ is a double covering map.
}\label{fig:X1/4}
\end{figure}
As discussed above, if $M_{ab}$ is even, then $X(a,b)_{\sfrac{1}{4}}$ has two connected components. The following lemma characterizes the graphs $\Gamma$ for which $X_{\sfrac 14}$ is connected.
\begin{lem}\label{lem:connected components of X14}
 The graph $X_{\sfrac 14}$ has either one or two connected components.  The following are equivalent:
\begin{itemize}
\item $X_{\sfrac 14}$ has two connected components,
\item each connected component of $X_{\sfrac 14}$ is a copy of $X_{\sfrac 12}$,
\item $\Gamma$ is a bipartite graph with all labels even.
\end{itemize}
\end{lem}

\begin{proof}
By Description~\ref{lem:structure of X14}, $X_{\sfrac 14}$ is a double cover of $X_{\sfrac 12}$. 
Since $X_{\sfrac 12}$ is connected, 
$X_{\sfrac 14}$ can have at most two connected components. 
The equivalence of the first two conditions
follows directly from that fact. 
Let us prove that the third condition is equivalent to the first one.

Suppose $\Gamma$ is a bipartite graph with all labels even. 
Since all labels are even, 
$X_{\sfrac 14}$ is isomorphic to the graph $\Gamma'$ 
with all edges doubled, 
so it suffices to show that $\Gamma'$ is not connected.
Let $U\sqcup W$ be the two parts of $V(\Gamma)$, 
i.e.\ $U\sqcup W = V(\Gamma)$, and each edge of $\Gamma$ joins a vertex of $U$ with a vertex of $W$. 
Denote by $U_{\pm}, V_{\pm}$ the preimage in $V_{\pm}$ of $U,W$ respectively.
Then $U_+\sqcup W_-$ and $U_-\sqcup W_+$ are the vertex sets of the two connected components of $X_{\sfrac14}$.
Indeed, this is true by the description of paths in $\Gamma'$ in Description~\ref{lem:structure of X14}.

Now suppose that $\Gamma$ is not a bipartite graph with all labels even. 
That means either $\Gamma$ has an edge with an odd label, or there is an odd length cycle in $\Gamma$.
We show that in both cases there is a path joining vertices $a_+, a_-$ in $X_{\sfrac 14}$ 
for some $a\in V(\Gamma)$ (equivalently any, again by Description~\ref{lem:structure of X14}).
If $\{a,b\}\in E(\Gamma)$ with $M_{ab}$ odd, 
then there is a path with vertices $a_+, b_-, a_-$ in $X_{\sfrac 14}$.
If $(a_1,\dots, a_{2n+1})$ is an odd length cycle in $\Gamma$, 
then by Description~\ref{lem:structure of X14} its lift to $\Gamma'$ contains a path joining $a_{1+}, a_{1-}$ as a subpath.
\end{proof}

%%%%%%%%%%%%%%%%%%%%%%%%%%
\subsection{The map $X_{\sfrac{1}{4}}\to X_{\sfrac{1}{2}}$}\label{sec:CtoB}
%%%%%%%%%%%%%%%%%%%%%%%%%%

As in Description~\ref{lem:structure of X14}, the map $X_{\sfrac{1}{4}}\to X_{\sfrac{1}{2}}$ factors as the composition of the inclusion $X_{\sfrac{1}{4}}\hookrightarrow N_{\sfrac{1}{2}}$ 
with the deformation retraction $N_{\sfrac{1}{2}}\to X_{\sfrac{1}{2}}$.
Under that map, $a_{\pm}\in V(X_{\sfrac 14})$ is mapped to $a\in V(X_{\sfrac{1}{2}})$ and every edge $\{a_{\pm}, b_{\pm}\}\in E(X_{\sfrac{1}{4}})$ is mapped to $\{a,b\}\in E(X_{\sfrac{1}{2}})$. 

If $\Gamma$ is a bipartite graph with all label even, 
then by Lemma~\ref{lem:connected components of X14}, $X_{\sfrac 14}$ is a disjoint union of two copies of $X_{\sfrac 12}$ 
and the map $X_{\sfrac{1}{4}}\to X_{\sfrac{1}{2}}$ is the identity map 
while restricted to each of the connected components. 
Otherwise, by Lemma~\ref{lem:connected components of X14}, $X_{\sfrac 14}$ is a connected double cover of $X_{\sfrac 12}$.
Then $C=\pi_1X_{\sfrac 14}\to B= \pi_1X_{\sfrac 12}$ is an inclusion of an index $2$ subgroup. 
The quotient $B/C = \Z/2\Z$ can be identified with the automorphism group of the covering space $X_{\sfrac{1}{4}}$ over $X_{\sfrac 1 2}$. 
In the case of $3$-generator Artin group this automorphism can be viewed as a $\pi$-rotation of the graph $X_{\sfrac{1}{4}}$ (with respect to the planar representation as in Figure~\ref{fig:X1/4}).

%%%%%%%%%%%%%%%%%%%%%%%%%%
\subsection{The map $X_{\sfrac{1}{4}} \to X_0$}\label{sec:CtoA}
%%%%%%%%%%%%%%%%%%%%%%%%%%
In this section we analyze the map $X_{\sfrac{1}{4}} \to X_0$ 
which is obtained by composing the inclusion $X_{\sfrac{1}{4}} \hookrightarrow N_0$
with the deformation retraction $N_0\to X_0$. 
This map is never combinatorial, and it might identify two distinct edges with the same origin.
This map, unlike $X_{\sfrac{1}{4}} \to X_{\sfrac 12}$, depends on the partial orientation $\iota$ of $\Gamma$.
In this section we give a description of the map  $X_{\sfrac{1}{4}} \to X_{0}$, and in the next section, we characterize when this map is $\pi_1$-injective,  in terms of the combinatorics of $\Gamma$ and $\iota$. 

In order to understand the map $X_{\sfrac{1}{4}} \to X_0$ 
we express it as a composition $X_{\sfrac{1}{4}}\to \overline X_{\sfrac{1}{4}}  \to X_0$ 
where the first map collapses some edges to a point and subdivides some other edges,
 and the second one is a combinatorial map.  
Proposition~\ref{prop:homotopic X14} gives conditions on $\Gamma$ for $X_{\sfrac{1}{4}} \to \overline X_{\sfrac{1}{4}}$ to be a homotopy equivalence, 
and so to be $\pi_1$-injective.
Proposition~\ref{prop:locally injective} gives conditions for $\overline X_{\sfrac{1}{4}}\to X_0$ to be a combinatorial immersion and consequently, $\pi_1$-injective.

The graph  $\overline X_{\sfrac{1}{4}}$ is obtained from $X_{\sfrac{1}{4}}$ in two steps: 
\begin{enumerate}
\item An edge that is sent to a vertex in $X_0$ is collapsed to a vertex, which results in identification of its endpoints. 
The edges that get collapsed are certain edges of $E'$. 
\item An edge that is sent to a single edge of $X_0$ via a degree $m$ map, is subdivided into a path of length $m$. The edges that get subdivided are certain edges of $E''$. 
 \end{enumerate}
We know from Section~\ref{sec:X14}
that $X(a,b)_{\sfrac{1}{4}}$ is a $4$-cycle $(a_-,b_+, a_+, b_-, a_-)$ if $M_{ab}$ is odd, 
and a disjoint union of $2$-cycles $(a_-, b_+, a_-)$ and $(a_+, b_-, a_+)$ if $M_{ab}$ is even. 
In both cases $X(a,b)_{\sfrac{1}{4}}$ is mapped to a single loop in $X_0(a,b)$. 
In all cases the map $X(a,b)_{\sfrac{1}{4}} \to X(a,b)_0\simeq S^1$ is a degree $M_{ab}$ map. 
If $M_{ab} = 2m$, then the map has degree $m$ restricted to each of the connected components. 
Since an incoming edge $a$ and an outgoing edge $b$ are adjacent in $X(a,b)$, the edge $\{a_+, b_-\}$ of the graph $X(a,b)_{\sfrac{1}{4}}$ is collapsed to a point (see Figure~\ref{fig:horizontal}).
Similarly, looking at the degrees of the map $X(a,b)_{\sfrac{1}{4}} \to X(a,b)_0$ restricted to other edges we find how to subdivide these edges to ensure that the map becomes combinatorial.
Let $\overline X(a,b)_{\sfrac{1}{4}}$ be the image of $X(a,b)_{\sfrac{1}{4}}$ in $\overline X_{\sfrac{1}{4}}$.

\begin{desc}\label{lem:collapsing an edge} 
The graph $\overline X(a,b)_{\sfrac{1}{4}}$ is obtained from $X(a,b)_{\sfrac{1}{4}}$ by:
\begin{itemize}
\item ($M_{ab}  = 2m+1$, $\iota(\{a,b\}) = a$): collapsing edge $\{a_+, b_-\}$ from $E(a,b)'$, 
and subdividing each of edges $\{a_+,b_+\}$ and $\{a_-, b_-\}$ from $E(a,b)''$ into a path of length $m$. 
Consequently, $\overline X(a,b)_{\sfrac{1}{4}}$ is a cycle of length $2m+1$.
\item ($M_{ab}  = 2$): collapsing edges $\{a_+, b_-\}$ and $\{a_-, b_+\}$ from $E(a,b)'$. 
Consequently, $\overline X(a,b)_{\sfrac{1}{4}}$ is a disjoint union of two length $1$ loops.
\item ($M_{ab}  = 2m\geq 4$, $\iota(\{a,b\}) = a$): collapsing edge $\{a_+, b_-\}$ from $E(a,b)'$, 
subdividing the edge $\{a_+, b_-\}$ from $E_{\sfrac14}''$ into a path of length $m$, and
subdividing the edge $\{a_-, b_+\}$ from $E_{\sfrac14}''$ into a path of length $m-1$. 
Consequently, $\overline X(a,b)_{\sfrac{1}{4}}$ is a disjoint union of cycles of length $m$ each.
\end{itemize}
\end{desc}

The above description of $\overline X(a,b)_{\sfrac{1}{4}}$ gives us the following description of $\overline X_{\sfrac{1}{4}}$.
See Figure~\ref{fig:horizontal graph maps} for an example of the factorization of the map $X_{\sfrac{1}{4}}\to X_0$ as a homotopy equivalence $X_{\sfrac{1}{4}} \to \overline X_{\sfrac{1}{4}}$ and a combinatorial map $\overline X_{\sfrac{1}{4}}\to X_0$. In that example $M_{ab} = 3, M_{bc} = 7, M_{ca} = 5$.
\begin{figure}
\includegraphics[scale=0.25]{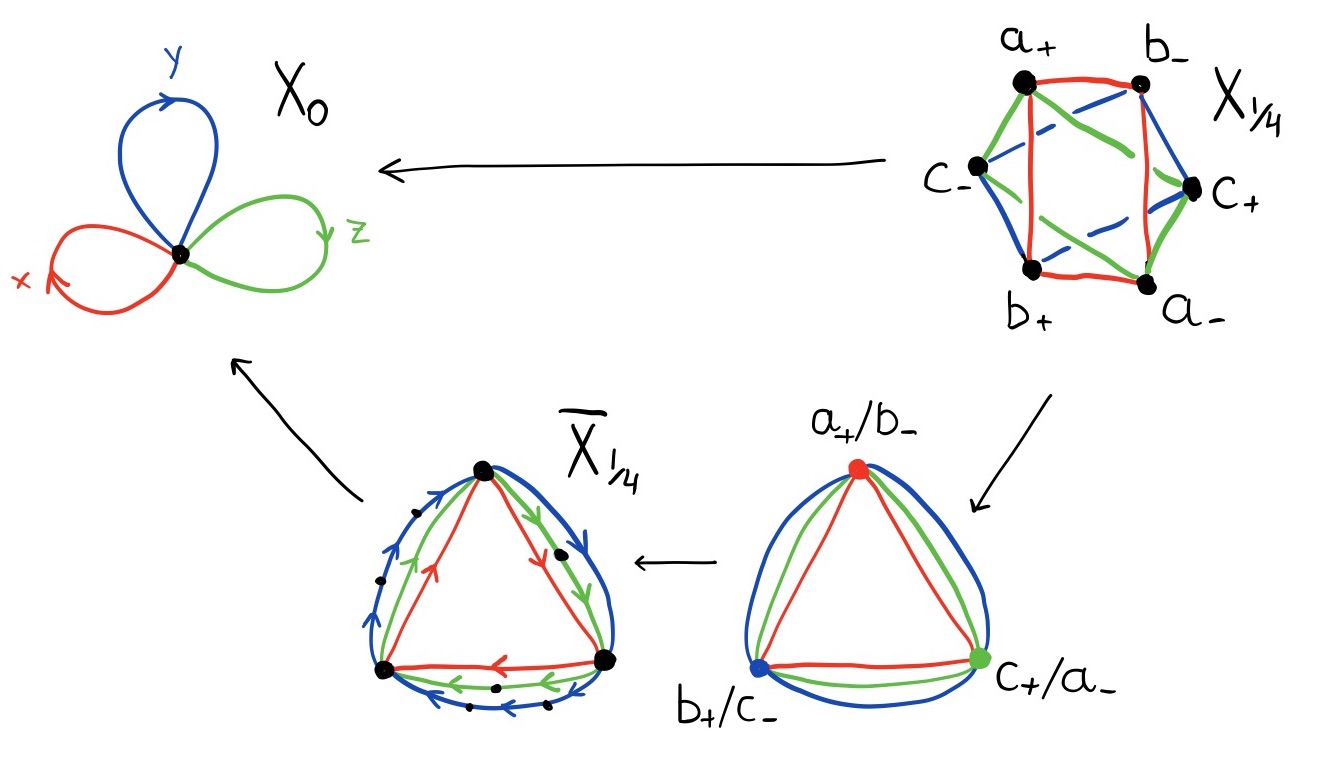}\caption{The graph $X_0, X_{\sfrac{1}{4}}, \overline X_{\sfrac{1}{4}}$ in $\Art_{375}$. The map $\overline X_{\sfrac{1}{4}}\to X_0$ is combinatorial.}\label{fig:horizontal graph maps}
\end{figure}
\begin{desc}\label{prop:structure of X14 bar}
The graph $\overline X_{\sfrac{1}{4}}$ can be described in terms of $\Gamma$ and $\iota$ as follows.
\begin{itemize}
\item There are two kinds of vertices in $\overline X_{\sfrac{1}{4}}$. 
Let $V_{old}$ denote the set of vertices that are the images of vertices in $X_{\sfrac{1}{4}}$, 
and $V_{new}$ consists of all other vertices that are introduced in the subdivision. 
There are two kinds of edges $\overline E',\overline E''$.
\item The vertices in $V_{old}$ correspond to the equivalence classes of $
V(X_{\sfrac{1}{4}}) =  V_+\sqcup V_-$ 
where the equivalence relation is generated by $a_+\sim b_-$ for every $\{a,b\}\in E(\Gamma)$ and every $a = \iota(\{a,b\})$ or $M_{ab} = 2$. 
\item The edges in $\overline E'$ are identified with the set 
$E' - \{\{a_+, b_-\}\in E'\mid \{a,b\}\in E(\Gamma) \text{ with }a = \iota(\{a,b\}) \text{ or } M_{ab}=2\}$, i.e.\
$\overline E'$ is the collection of all edges of $E'$ that do not get collapsed.
\item For each edge $\{a_{\pm},b_{\pm}\}$ in $E''$ there is a path of length $m$ or $m-1$ as in Descrription~\ref{lem:collapsing an edge}, consisting of edges of $\overline E''$ and joining appropriate vertices in $V_{old}$. 
The vertices inside such paths form the set $V_{new}$.
\end{itemize}
Each $\overline X(a,b)_{\sfrac14}$ admits a natural combinatorial immersion onto a corresponding loop of $X_0$. 
The map $\overline X_{\sfrac 14}\to X_0$ is defined piecewise using these maps $\overline X(a,b)_{\sfrac14}\to X_0$. In the next subsection, we characterize when $\overline X_{\sfrac 14}\to X_0$ is $\pi_1$-injective.
\end{desc}

%\begin{proof}
%Follows from Lemma~\ref{lem:collapsing an edge}.
%\end{proof}

%%%%%%%%%%%%%%%%%%%%%%%%%%
\subsection{Conditions for $\pi_1$-injectivity of $X_{\sfrac{1}{4}} \to X_0$}
%%%%%%%%%%%%%%%%%%%%%%%%%%
We are now ready to characterize when the map $X_{\sfrac 14}\to X_0$ is $\pi_1$-injective. 
The next two propositions ensure that the maps $X_{\sfrac 14}\to \overline X_{\sfrac 14}$ 
and $\overline X_{\sfrac 14} \to X_0$ respectively, are $\pi_1$-injective. 
We refer to Definition~\ref{defn:misdirected} for the definition of an (almost) misdirected cycles.
We emphasize that a misdirected even length cycle in the statement of Proposition~\ref{prop:homotopic X14} 
is not assumed to be simple.

\begin{prop}\label{prop:homotopic X14} 
Let $\Gamma$ be a simple graph with edges labelled by numbers $\geq 2$ 
with a partial orientation $\iota$ such that 
$\iota(e)$ of an edge $e$ is defined if and only if the label of $e$ is $\geq 3$. 
Then $X_{\sfrac{1}{4}} \to \overline X_{\sfrac{1}{4}}$ is a homotopy equivalence 
if and only if $\Gamma$ has no misdirected even length cycles 
and no cycles with all edges labelled by $2$.
\end{prop}

\begin{proof}
We refer to Description~\ref{lem:structure of X14} for the structure of $X_{\sfrac 14}$ 
and to Description~\ref{prop:structure of X14 bar} for the structure of $\overline X_{\sfrac 14}$. 
By construction, $X_{\sfrac{1}{4}} \to \overline X_{\sfrac{1}{4}}$ is obtained by collapsing certain edges of $X_{\sfrac{1}{4}}$ followed by edge subdivision. 
The edge subdivision never changes the homotopy type of a graph, but edge collapsing might.
The map $X_{\sfrac{1}{4}} \to \overline X_{\sfrac{1}{4}}$ fails to be a homotopy equivalence
if and only if there is a cycle in $X_{\sfrac{1}{4}}$ with all edges collapsed in $\overline X_{\sfrac{1}{4}}$

First let us assume that $\Gamma$ has a misdirected even length cycle $(a_1, a_2, \dots, a_{n}, a_1)$. 
Then each of the edges in one of the cycles 
$(a_{1+}, a_{2-}, \dots, a_{n-}, a_{1+})$ or $(a_{1-}, a_{2+}, \dots, a_{n+}, a_{1-})$ 
of $\Gamma'\subseteq X_{\sfrac{1}{4}} $ gets collapsed. 
Thus, $X_{\sfrac{1}{4}} \to \overline X_{\sfrac{1}{4}}$ is not a homotopy equivalence.
Now suppose that $\Gamma$ has an odd length cycle $(a_1, a_2, \dots, a_{n}, a_1)$ with all edges labelled by $2$.
Then its lift to $\Gamma'$ is the cycle 
$(a_{1+}, a_{2-}, \dots, a_{n+}, a_{1-}, \dots, a_{n-}, a_{1+})$ 
and all of its edges get collapsed in $\Gamma'\subseteq \overline X_{\sfrac 14}$.

Conversely, suppose that there exists a cycle $\gamma'$ in $X_{\sfrac 1 4}$ all of whose edges are collapsed in $\overline X_{\sfrac 1 4}$.
%$\Gamma$ has no misdirected even length cycles and no cycles with all edges labelled by $2$. 
%It suffices to show that every cycle in $X_{\sfrac 1 4}$ contains at least one edge 
%that is not collapsed in $\overline X_{\sfrac 1 4}$. 
Only edges from the set $E'$ might be collapsed by Description~\ref{prop:structure of X14 bar}, 
so $\gamma'\subseteq \Gamma'$. 
Without loss of generality, we can assume that $\gamma'$ is a simple cycle. 
Let $\gamma$ be the image of $\gamma'$ in $\Gamma$.

First suppose $\gamma$ is not simple. 
Then there exists a vertex $a\in \Gamma$ such that $\gamma'$ passes through both $a_-$ and $a_+$, 
i.e.\ $\gamma'$ can be expressed as $(a_-, b_{1+}, b_{2-}, \dots, b_{k-}, a_+, c_{1-}, c_{2+}, \dots, c_{l+}, a_-)$ for some $k,l$. Since the plus and minus signs in labels of $\gamma$ alternate, the numbers $k,l$ must be even.
If each edge of $\gamma'$ is collapsed, that means that the partial orientation on $\gamma$ induced by $\iota$ extends to an orientation on $\gamma= (a, b_1, b_2, \dots, b_{k}, a, c_1, c_2, \dots, c_{l}, a)$ as pictured in Figure~\ref{fig:misdirected cycle}. 
\begin{figure}
\includegraphics[scale=0.25]{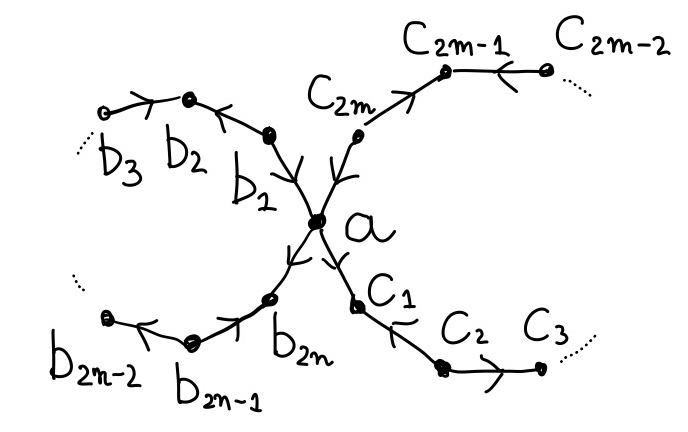}
\caption{The misdirected non-simple cycle $\gamma$.}\label{fig:misdirected cycle}
\end{figure}
In particular, $\gamma$ is a misdirected even length (non-simple) cycle. 
%
%Each of the subpath of $\gamma'$ joining $a_-$ and $a_+$ must have odd length by Description~\ref{lem:structure of X14}, and $\gamma$ is a non-simple even length cycle. By assumption, $\gamma$ is not misdirected, so at least one of the edges of $\gamma'$ does not get collapsed.
%

Now suppose that $\gamma$ is simple.
Either $\gamma'\to\gamma$ is two-to-one or one-to-one, depending on the parity of the length of $\gamma$.
If the length of $\gamma$ is odd, then $\gamma'\to \gamma$ is two-to-one. 
%By assumption $\gamma$ contains an edge $\{a,b\}$ whose label is not $2$. 
For every edge $\{a,b\}$ in $\gamma$, both edges $\{a_-, b_+\}$ and $\{a_+, b_-\}$ are contained in $\gamma'$, by  Description~\ref{lem:structure of X14}. Moreover, if they both get collapsed, that means that $M_{ab}=2$. Thus $\gamma$ is a cycle with all edges labelled by $2$.
%, $\gamma'$ contains both edges $\{a_-, b_+\}$ and $\{a_+, b_-\}$ 
%and, by Description~\ref{prop:structure of X14 bar}, only one of these two edges gets collapsed.
If the length of $\gamma$ is even, then $\gamma'\to \gamma$ is one-to-one.
This necessarily means that $\gamma$ is a misdirected even length cycle.
%then  by Description~\ref{lem:structure of X14}, 
%and again at least one of the edges of $\gamma'$ does not get collapsed. 
%If $\gamma'\to \gamma$ is two-to-one, then $\gamma$ has odd length and 
%by assumption contains an edge $\{a,b\}$ whose label is not $2$. 
%Then, by Description~\ref{lem:structure of X14}, $\gamma'$ contains both edges $\{a_-, b_+\}$ and $\{a_+, b_-\}$ 
%and, by Description~\ref{prop:structure of X14 bar}, only one of these two edges gets collapsed.
\end{proof}

\begin{prop}\label{prop:locally injective} 
Let $\Gamma$ be a simple graph with all labels $\geq 2$ with a partial orientation $\iota$ such that 
$\iota(e)$ of an edge $e$ is defined if and only if the label of $e$ is $\geq 3$.
Suppose $\Gamma$ has no misdirected even length cycles. % and cycles with all labels $2$.
Then $\overline X_{\sfrac{1}{4}}\to X_0$ is a combinatorial immersion 
if and only if $\Gamma$ has no almost misdirected cycles.
\end{prop}

\begin{proof} 
By construction $\overline X_{\sfrac{1}{4}}\to X_0$ is always a combinatorial map. It fails to be an immersion precisely when there are more than one oriented edges at some vertex of $\overline X_{\sfrac{1}{4}}$ mapping to the same oriented edge of $X_0$. 

Recall that the edges of $X_0$ are in one-to-one correspondence with edges of $E(\Gamma)$, and 
for each edge $x$ of $X_0$ the edges of $X_{\sfrac{1}{4}}$ mapping to $x$ are exactly those 
coming from a single $X(a,b)_{\sfrac{1}{4}}$ for some $\{a,b\}\in E(\Gamma)$. 
Recall that every graph $X(a,b)_{\sfrac{1}{4}}$ is a single cycle if $M_{ab}$ is odd, and a union of two cycles if $M_{ab}$ is even. 
We claim that the map $\overline X_{\sfrac{1}{4}}\to X_0$ is not an immersion if and only if there exists a path in $\Gamma'$ joining two vertices $v_1, v_2$ of $X(a,b)_{\sfrac{1}{4}}$ such that 
\begin{itemize}
\item $v_1, v_2$ are not identified within $\overline X(a,b)_{\sfrac{1}{4}}$ as described in Description~\ref{lem:collapsing an edge},
\item the path gets entirely collapsed in $\overline X_{\sfrac{1}{4}}$.
\end{itemize}
Indeed, if such path exists then $v_1, v_2$ project to distinct vertices $\bar v_1, \bar v_2\in\overline X(a,b)_{\sfrac{1}{4}}$. 
Each $\bar v_i$ is adjacent to the unique oriented edge $e_i$ that maps onto the oriented edge $x$ in $X_0$. 
By the second condition $\bar v_1, \bar v_2$ become identified in $\overline X_{\sfrac{1}{4}}$. 
However, the edges $e_1, e_2$ remain distinct in $\overline X_{\sfrac{1}{4}}$, and therefore $\overline X_{\sfrac{1}{4}}\to X_0$ is not an immersion. 
Conversely, if there are two oriented edges $\bar e_1, \bar e_2$ in $\overline X_{\sfrac{1}{4}}$ that maps onto the oriented edge $x$ in $X_0$, then $\bar e_1, \bar e_2$ are images of distinct oriented edges $e_1, e_2$ in $\overline X(a,b)_{\sfrac{1}{4}}$. The initial vertices of $e_1, e_2$ must be distinct vertices in $\overline X(a,b)_{\sfrac{1}{4}}$ which become identified in $\overline X_{\sfrac{1}{4}}$. Thus their preimages in $X_{\sfrac{1}{4}}$ must be connected by a path as above.

We first prove that if $\Gamma$ contains an almost misdirected cycle, then $\overline X_{\sfrac{1}{4}}\to X_0$ is not a combinatorial immersion.
 An odd length almost misdirected cycle $\gamma = (a_1, a_2, \dots, a_n, a_1)$ in $\Gamma$ where the path $(a_1, a_2, \dots, a_n)$ is misdirected, yields a path $\gamma'$ in $\Gamma'$ joining either $a_{1-}, a_{1_+}$ or $a_{n-}, a_{n+}$ that gets entirely collapsed in $\overline X_{\sfrac{1}{4}}$. 
 Neither the pair $a_{1-}, a_{1_+}$ or $a_{n-}, a_{n+}$ is identified within any copy of $X(a,b)_{\sfrac{1}{4}}$. 
 Thus the map $\overline X_{\sfrac{1}{4}}\to X_0$ is not an immersion.
  
Now suppose that $\gamma = (a_1,\dots, a_n, a_1)$ is an even length almost misdirected cycle where the path $(a_1, a_2, \dots, a_n)$ is misdirected. By the assumption, $\gamma$ is not misdirected and
%Without loss of generality by possibly replacing it by $(a_n,\dots, a_1)$, 
we can assume (by possibly replacing $\gamma$ with the cycle $(a_1, a_n,\dots, a_1)$) that there exists an orientation $\bar \iota$ extending the partial on $\gamma$ induced by $\iota$ such that $\bar \iota (\{a_1, a_2\}) = a_2$, $\bar \iota (\{a_2, a_3\}) = a_2, \dots, \bar \iota (\{a_{n-1}, a_n\}) = a_n,$ and $\bar \iota (\{a_1, a_n\}) = a_1$. 
Then the path $a_{1-}, a_{2+}, \dots  a_{n_+}$ gets entirely collapsed. 
The vertices $a_{1-}, a_{n_+}$ become identified but the edge $\{a_{1-}, a_{n_+}\}$ of $\Gamma'$ was not collapsed in $\overline  X(a_1,a_n)_{\sfrac{1}{4}}$. That means that the map $\overline X_{\sfrac{1}{4}}\to X_0$ not a combinatorial immersion.

Conversely, suppose that there exists a path $\gamma'$ in $X_{\sfrac{1}{4}}$ that joins one of $a_-,b_+$ with one of $a_+, b_-$ with all edges getting collapsed in $\overline X_{\sfrac{1}{4}}$ such that $\gamma'$ is not a single edge of $X(a,b)_{\sfrac{1}{4}}$.
%two edges of $\overline X_{\sfrac{1}{4}}$ that map to $x$.
%
First consider the case where $M_{ab} = 2$.
%For $ X(a,b)_{\sfrac{1}{4}}$ to be mapped to $X_0$ in a non-locally injective manner, there must be a path $\gamma'$ in $X_{\sfrac{1}{4}}$ that joins one of $a_-,b_+$ with one of $a_+, b_-$ with all edges getting collapsed in $\overline X_{\sfrac{1}{4}}$. 
Since only certain edges from the set $E'$ get collapsed we can assume that $\gamma'\subseteq \Gamma'$. 
Without loss of generality by possibly extending $\gamma'$ by extra edges $\{a_-, b_+\}$ or $\{a_+, b_-\}$, we can assume that $\gamma'$ joins $a_-$ and $a_+$.
Then $\gamma'$ projects to $\gamma \in \Gamma$, which is an odd length almost misdirected cycle. 
See Figure~\ref{fig:almost misdirected combinatorial immersion} (left).
\begin{figure}
\includegraphics[scale=0.25]{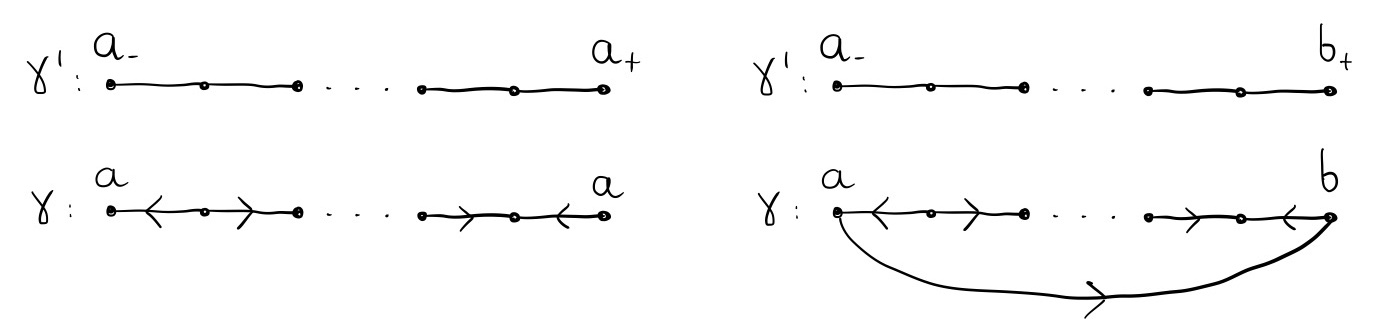}
\caption{The orientation on $\gamma$ extending the partial orientation induced by $\iota$.}\label{fig:almost misdirected combinatorial immersion}
\end{figure}

Now suppose $M_{ab}\neq 2$, and let $\iota(\{a,b\}) = a$, i.e.\ the edge $\{a_+, b_-\}$ of $\Gamma'$ is collapsed but the edge $\{a_-, b_+\}$ remains uncollapsed in $\overline X_{\sfrac{1}{4}}$.
If there is a path $\gamma'$ in $\Gamma'$ joining $a_-$ and $a_+$, or $b_-$ and $b_+$ with all edges getting collapsed,
then the argument above again gives an odd length almost misdirected cycle in $\Gamma$. 
Otherwise, %for $X(a,b)_{\sfrac{1}{4}}$ to be mapped to $X_0$ in a non-locally injective manner, 
there must be a path $\gamma'$ in $X_{\sfrac{1}{4}}$ joining $a_-$ with $b_+$. 
Then $\gamma'$ projects to $\gamma \in \Gamma$, 
such that $\gamma\cup\{a,b\}$ is an even length almost misdirected cycle (which is not misdirected).
See Figure~\ref{fig:almost misdirected combinatorial immersion} (right).

\end{proof}

%%%%%%%%%%%%%%%%%%%%%%%%%%
\subsection{Proof of the Splitting Theorem}\label{sec:decomposition}
%%%%%%%%%%%%%%%%%%%%%%%%%%
We are finally ready to prove Theorem~\ref{thm:decomposition}.

\begin{proof}[Proof of Theorem~\ref{thm:decomposition}]
By Proposition~\ref{prop:homotopic X14} and Proposition~\ref{prop:locally injective}, 
the map $X_{\sfrac 14}\to X_0$ factors as a composition of a homotopy equivalence and a combinatorial immersion,
and thus is $\pi_1$-injective. 
By Lemma~\ref{lem:connected components of X14}, $X_{\sfrac 14}$ is connected 
if and only if $\Gamma$ is not a bipartite graph with all labels even.
In such case, the conclusion follows from Lemma~\ref{lem:pushout}.
If $\Gamma$ is a bipartite graph with all labels even, 
then by Lemma~\ref{lem:connected components of X14} $X_{\sfrac 14}$ has two connected components, 
and each of them is a copy of $X_{\sfrac 12}$.
By Lemma~\ref{lem:hnn} $\Art_{\Gamma}$ splits as an HNN-extension $A*_B$ where $A = \pi_1 X_0$ and $B=\pi_1X_{\sfrac 12}$. 
If $\Gamma$ is not a bipartite graph with all labels even, then by Lemma~\ref{lem:pushout} $\Art_{\Gamma}$ splits as $A*_CB$ where $C=\pi_1 X_{\sfrac 14}$. 
Since $X_0$ is a bouquet of $|E(\Gamma)|$ loops, $\rank A = |E(\Gamma)|$.
The graph $X_{\sfrac 12}$ is a copy of $\Gamma$ with doubled edges, 
so $\euler(X_{\sfrac{1}{2}}) = |V(\Gamma)| - 2|E(\Gamma)|$.
Hence $\rank B=1 - |V(\Gamma)| + 2|E(\Gamma)|$.
In the case of amalgamated product, $\rank C = 2\rank B -1 = 1 -2|V(\Gamma)| +4|E(\Gamma)|$, since  the index of $C$ in $B$ is two.
\end{proof}

%%%%%%%%%%%%%%%%%%%%%%%%%%
%\subsection{Twisted double structure}\label{sec:twisted double}
%%%%%%%%%%%%%%%%%%%%%%%%%%

\begin{remark}[Twisted double of free groups as index two subgroup of $G$]\label{rem:Artin twisted double}
Let $G$ be any amalgamated product $A*_CB$ of groups such that the index of $C$ in $B$ is two. 
Let $g$ be a representative of the nontrivial coset of $B/C$ and 
denote by $\beta:C\to C$ the automorphism given by $\beta(h) = g^{-1}hg$.
Since $g^2\in C$, $\beta^2$ is an inner automorphism of $C$.
The group $G = A*_CB$ has an index two subgroup 
isomorphic to the twisted double $D(A, C,\beta)$, which is the kernel of the homomorphism $G \to B/C$. 

In particular, every $\Art_{\Gamma}$ that splits as an amalgamated product as in Theorem~\ref{thm:decomposition}
has an index two subgroup $D(A,C,\beta)$. 
Geometrically, $\beta$ is a nontrivial deck transformation of the graph $X_{\sfrac 1 4}$ 
as a covering space of $X_{\sfrac 1 2}$.
In the case of the three generator $\Art_{\Gamma}$, $\beta$ can be viewed as a rotation by  $\pi$ 
(with respect to the planar representation in Figure~\ref{fig:X1/4}). 
The choice of the element $g\in B-C$ corresponds to 
the choice of a path joining a basepoint in $X_{\sfrac 1 4}$ with the opposite vertex 
(e.g. $a_+$ with $a_-$).
\end{remark}

\subsection{Explicit splittings for $3$-generator Artin groups}
%%%%%%%%%%%%%%%%%%%%%%%%%%
Let us now explicitly describe the splitting in Theorem~\ref{thm:decomposition} in the case of large type Artin group where $\Gamma$ is a triangle.

\begin{cor}\label{cor:splitting of 3 gen}
Let $\Art_{MNP}$ be an Artin group where $M,N,P\geq 3$.  Then $\Art_{MNP} = A*_CB$ where $A\simeq F_3$, $B\simeq F_4$ and $C\simeq F_7$, and $[B:C] = 2$. The map $C\to A$ is induced by the maps pictured in Figure~\ref{fig:mapCtoA}.
\begin{figure}
\includegraphics[scale=0.25]{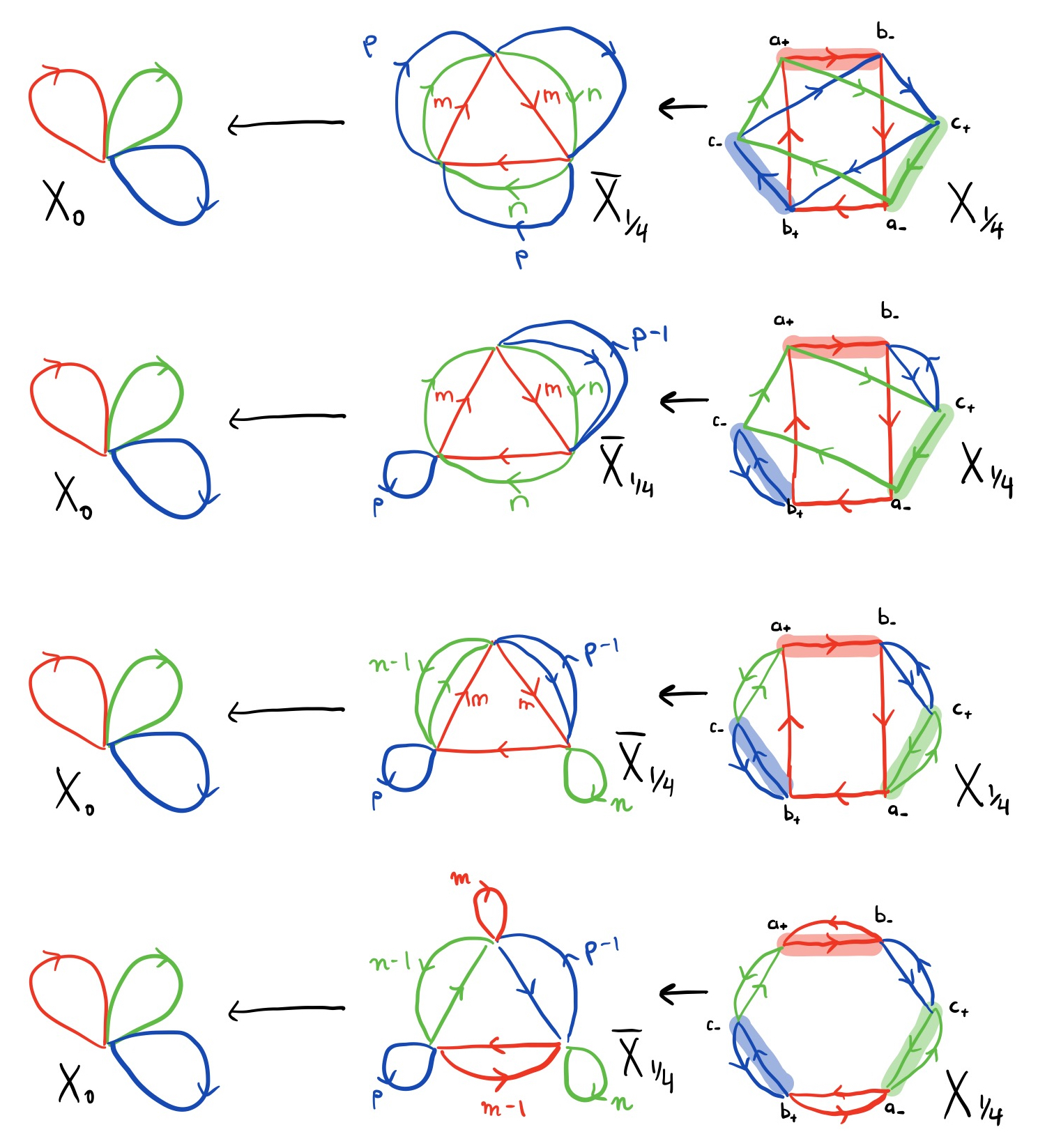}
\caption{The map $X_{\sfrac{1}{4}}\to X_0$ when (1) none, (2) one, (3) two or (4) all of $M_{ab} = M,M_{bc} = N, M_{ca} =P$ are even, respectively.
Specifically, $M=2m$ or $2m+1$, $N=2n$ or $2n+1$, and $P=2p$ or $2p+1$. 
Here we use the convention where the edge labelled by a number $k$ is a concatenation of $k$ edges of the given color. The distinguished edges in $X_{\sfrac{1}{4}}$ are the ones that get collapsed to a vertex in $\overline X_{\sfrac{1}{4}}$.}\label{fig:mapCtoA}
\end{figure}
Moreover, $\Art_{MNP}$ has an index two subgroup that is isomorphic to the twisted double $D(A,C,\beta)$ where $\beta:C\to C$ is given by $\beta(h) = g^{-1}hg$ for some (equivalently any) $g\in B-C$.
\end{cor}

\begin{proof}
Since $\Gamma$ is a triangle, we have $|V(\Gamma)| = |E(\Gamma)| = 3$. 
By ordering $\Gamma$ cyclically, we obtain a graph without misdirected cycles. 
By Theorem~\ref{thm:decomposition}, $\Art_{\Gamma}$ splits as $A*_C B$ 
where $\rank A = 3$, $\rank B = 1-3 +2*3 = 4$ and $\rank C = 2*4-1 = 7$. 
The maps $X_{\sfrac 14} \to X_0$ inducing $A\to C$ in Figure~\ref{fig:mapCtoA} 
come directly from the descriptions in Section~\ref{sec:CtoA}. 
The index two subgroup isomorphic to a twisted double comes from Remark~\ref{rem:Artin twisted double}.
\end{proof}

\begin{exa}[$\Art_{333}$]\label{exa:Art333}
By Corollary ~\ref{cor:splitting of 3 gen}, $\Art_{333}$ splits as $F_3*_{F_7}F_4$ 
and the map $\overline X_{\sfrac1 4}\to X_0$ is a regular cover of degree $3$. 
See the top of Figure~\ref{fig:mapCtoA} with $m=n=p=1$. 
Thus $C\simeq F_7$ is a normal subgroup in each of the factors and $[C:A] = 3$. 
This splitting of $\Art_{333}$ as $F_3*_{F_7}F_4$ was first proved in \cite{Squier87}. 
We have the following short exact sequence
\[
1\to F_7 \to \Art_{333}\to \Z/3*\Z/2\to 1.
\]
We conclude that $\Art_{333}$ is (fin.~rank free)-by-(virtually fin.~rank free), 
and therefore virtually (fin.~rank free)-by-free. 
In particular, $\Art_{333}$ is virtually a split extension of a finite rank free group by a free group. 
Since every split extension of a finitely generated residually finite group 
by residually finite group is residually finite \cite{Malcev83}, 
$\Art_{333}$ is residually finite. 
\end{exa}

%%%%%%%%%%%%%%%%%%%%%%%%%%%%%%%%%%%%%%%%%%%%%%%%%%%%
%%%%%%%%%%%%%%%%%%%%%%%%%%%%%%%%%%%%%%%%%%%%%%%%%%%%
\section{Residual finiteness of 3-generator Artin groups}\label{sec:rf of 3 gen}
%%%%%%%%%%%%%%%%%%%%%%%%%%%%%%%%%%%%%%%%%%%%%%%%%%%%
%%%%%%%%%%%%%%%%%%%%%%%%%%%%%%%%%%%%%%%%%%%%%%%%%%%%
In this section, we prove Theorem~\ref{thm:main}. 
By Corollary~\ref{cor:splitting of 3 gen}, $\Art_{MNP}$ with $M,N,P\geq 3$ splits 
as a free product with amalgamation $A*_C B$ of finite rank free groups, 
and is virtually a twisted double $D(A,C,\beta)$. 
Throughout this section, $A,B,C$ are the groups from the splitting in Corollary~\ref{cor:splitting of 3 gen} and $M,N,P\geq 4$.
We begin with computing how far the subgroup $C$ is from being malnormal in $A$.
Then we prove Theorem~\ref{thm:main} (stated as Corollary~\ref{cor:at least one even} and Corollary~\ref{cor:all odd}) by applying Theorem~\ref{thm:conditions for rf}. 
In Section~\ref{sec:at least one even} we consider the easier case where at least one of $M,N,P$ is even 
and then in Section~\ref{sec:all exponents odd} we proceed with the case where $M,N,P$ are all odd.

%%%%%%%%%%%%%%%%%%%%%%%%%%
\subsection{Failure of malnormality} 
%%%%%%%%%%%%%%%%%%%%%%%%%%
A twisted double  $D(A,C,\beta)$ where $A,C$ are finite rank free groups and $C$ is malnormal in $A$ is hyperbolic by \cite{BestvinaFeighn92}. 
However, $\Art_{\Gamma}$ is never hyperbolic, unless $\Gamma$ is a single point, 
in which case $\Art_{\Gamma} = \Z$. 
Thus the intersection $C^g\cap C$ must be nontrivial for some $g\in A$. 
Understanding how the edge group $C$ intersects its conjugates plays a crucial role in our proof.

The intersection $C^g\cap C$ can be computed using the fiber product $\overline X_{\sfrac 1 4}\otimes_{X_0}\overline X_{\sfrac 1 4}$ (see Section~\ref{sec:fiber}).
The maps $\overline X_{\sfrac 14}\to X_0$ is described in Section~\ref{sec:CtoA} 
and pictured in Figure~\ref{fig:mapCtoA}. 

Let $F$ denote the fiber product $\overline X_{\sfrac 1 4}\otimes_{X_0}\overline X_{\sfrac 1 4}$.
The vertex set $V(F)$ is the product $V(\overline X_{\sfrac 14})\times V(\overline X_{\sfrac 14})$ 
and the edge set $E(F)$ is a subset of $E(\overline X_{\sfrac 1 4})\times E(\overline X_{\sfrac 1 4})$. 
All the nontrivial connected components of $F$ (i.e.\ the ones without vertices of the form $(v,v)$ for some $v\in V(\overline X_{\sfrac 14})$) correspond to some $C\cap C^g$ where $g\notin C$ by \cite{Stallings83}. 

Let $Y$ be either $\overline X_{\sfrac 1 4}$ or $\overline X_{\sfrac 1 4}\otimes_{X_0}\overline X_{\sfrac 1 4}$. 
We continue to represent the map $Y\to X_0 $ by coloring the edges of $Y$ where each color represents one of the edges of $X_0$.
We say a cycle or a path in $Y$ is \emph{monochrome}, if it is mapped onto a single loop in $Y$.

Note that any two simple monochrome cycles in $\overline X_{\sfrac 1 4}$ of the same color, have the same length.
Hence all the simple monochrome cycles lift to their copies in $F$. 
Thus any connected component of $F$ 
is a union of simple monochrome cycles whose lengths are the same as in $\overline X_{\sfrac 1 4}$. 
The branching vertices (i.e.\ of valence $>2)$ of connected components of $F$
are contained in $V_{old}\times V_{old}\subseteq V(\overline X_{\sfrac 14})\times V(\overline X_{\sfrac 14})$, 
since $V_{old}$ are the only branching vertices of $\overline X_{\sfrac 14}$. In particular, all the segments (i.e.\ paths between branching vertices with all internal vertices of valence $2$) in $F$ are monochrome.

\begin{lem}[All odd]\label{lem:fiber all odd} 
Suppose $(M,N,P) = (2m+1,2n+1,2p+1)$ where $m,n,p\geq 2$. 
Then the intersection $C^g\cap C$ for $g\in A-C$ is either trivial, 
or its conjugacy class is represented by a subgraph of the graph in Figure~\ref{fig:fiber product}.
\begin{figure}[h]\includegraphics[scale=0.3]{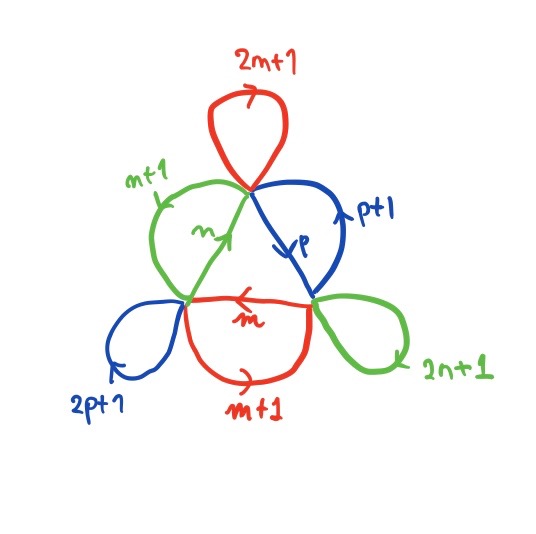}\caption{A non-trivial component of $F$, when $M,N,P$ are all odd. The vertex $(v_r, v_g)$ is the bottom left.}\label{fig:fiber product}\end{figure}
\end{lem}

\begin{proof}
This proof is a direct computation of the fiber product of $F$.
Let $\{r_0,\dots, r_{2m}\}$, $\{g_0,\dots, g_{2n}\}$ and $\{b_0,\dots, b_{2p}\}$ 
be the sets of cyclically ordered (consistently with the orientation of the cycle) 
vertices in $\overline X_{\sfrac 1 4}$ of red, green and blue cycle respectively 
such that $v_r:=r_0 = g_0 = b_0$, $v_g:=r_{m} = g_{n} = b_1$ and $v_b:=r_{m+1} = g_{2n} = b_{p+1}$ 
are in $V_{old}$. 
The vertices $v_r, v_g, v_b$ come from collapsing a red, green, blue edge of $X_{\sfrac 14}$ respectively. %(see Figure~\ref{fig:mapCtoA}). 
They are respectively the top, the bottom right and the bottom left vertices in $\overline X_{\sfrac 14}$ in Figure~\ref{fig:mapCtoA}.
The connected component containing vertices $(v_r, v_g), (v_g,v_b), (v_b,v_r)$ 
is illustrated in Figure~\ref{fig:fiber product}. 
Another copy of that graph is the connected component containing $(v_r, v_b), (v_g, v_r), (v_b,v_g)$. 
All other nontrivial connected components do not have any branching vertices, 
and so are single monochrome cycles, or single vertices. 
\end{proof}

%\begin{lem}[All odd, some equal $3$]\label{lem:fiber all odd} 
%Suppose $(M,N,P) = (2m+1,2n+1,2p+1)$ and $M,N,P\geq 5$. 
%Then the intersection $C^g\cap C$ for $g\in A-C$ is either trivial, 
%or its conjugacy class is represented by a subgraph of the graph in Figure~\ref{fig:fiber product all odd 2}.
%\begin{figure}[h]\includegraphics[scale=0.3]{intersectionofconjugates}\caption{A non-trivial component of $F$, when $M,N,P$ are all odd.}\label{fig:fiber product}\end{figure}
%\end{lem}
%
%\begin{remark}
%The assumption in Lemma~\ref{lem:fiber all odd} that at least one of $M,N,P$ is not equal $3$ is necessary. 
%As mentioned in Example~\ref{exa:Art333}, if $M=N=P=3$, then $\overline X_{\sfrac 14}$ is a regular cover. 
%In that case all (three) connected components of $F$ 
%are copies of $\overline X_{\sfrac 14}$.
%\end{remark}

\begin{lem}[One even]\label{lem:fiber one even} 
If $(M,N,P) = (2m+1,2n+1,2p)$ where $m,n,p\geq 2$, then the intersection $C^g\cap C$ for $g\in A$ is either trivial, 
or its conjugacy class is represented by a subgraph of the graph in Figure~\ref{fig:fiber product two odd}.
\begin{figure}[h]\includegraphics[scale=0.3]{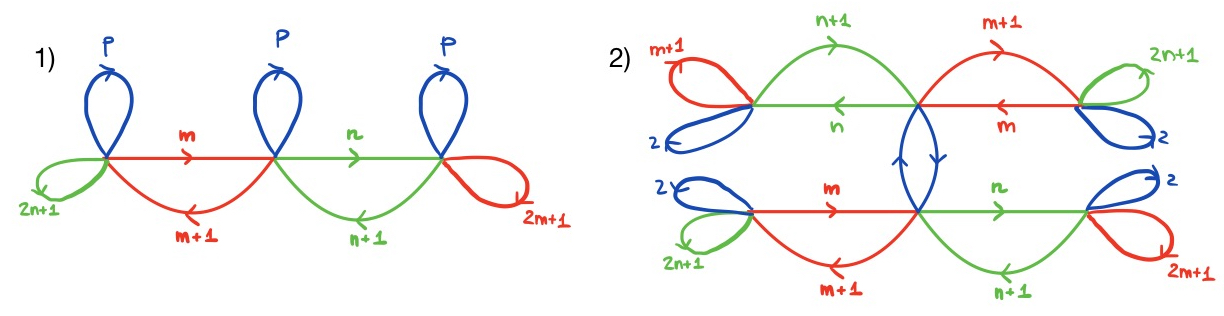}\caption{A non-trivial component of $F$, when $M,N$ are odd and $P$ is even. }\label{fig:fiber product two odd}\end{figure}
\end{lem}

\begin{proof}
We analyze the fiber product $F$ 
as in proof of Lemma~\ref{lem:fiber all odd}. 
Let $\{r_0,\dots, r_{2m}\}$ and $\{g_0,\dots, g_{2n}\}$ 
be the sets of cyclically ordered vertices of red and green cycle respectively, 
and and $\{b_0,\dots, b_{p-1}\}$ and $\{b_{p}\,\dots, b_{2p-1}\}$ 
be the sets of cyclically ordered vertices of the two blue cycles 
such that $v_r:=r_0 = g_0 = b_0$, $v_g:=r_{m} = g_{n} = b_1$ and $v_b:=r_{m+1} = g_{2n} = b_{p}$. 
As before the only branching vertices in $F$ are pairs of branching vertices of $\overline X_{\sfrac 14}$. 

If $p>2$, then $F$ has two connected components, one containing the vertices $(v_r,v_g), (v_g, v_b),(v_b,v_r)$ 
and one containing the vertices $(v_r, v_b), (v_g, v_r) ,(v_b,v_g)$. Each of them is a copy of the graph is illustrated in Figure~\ref{fig:fiber product two odd}(1). In the first case, the vertex $(v_r, v_g)$ is in the center.
All the connected components without branching vertices are simple monochrome cycles, or single vertices.

If $p=2$, then the vertices $(v_r, v_g)$ and $(v_g, v_r)$ are adjacent. In that case $F$ is connected and is illustrated in Figure~\ref{fig:fiber product two odd}(2).
\end{proof}

\begin{lem}[Two even]\label{lem:fiber two even} 
If $(M,N,P) = (2m+1,2n,2p)$ where $m,n,p\geq 2$, then the intersection $C^g\cap C$ for $g\in A$ is either trivial, 
or its conjugacy class is represented by a subgraph of the graph in Figure~\ref{fig:fiber product one odd}.
\begin{figure}[h]\includegraphics[scale=0.3]{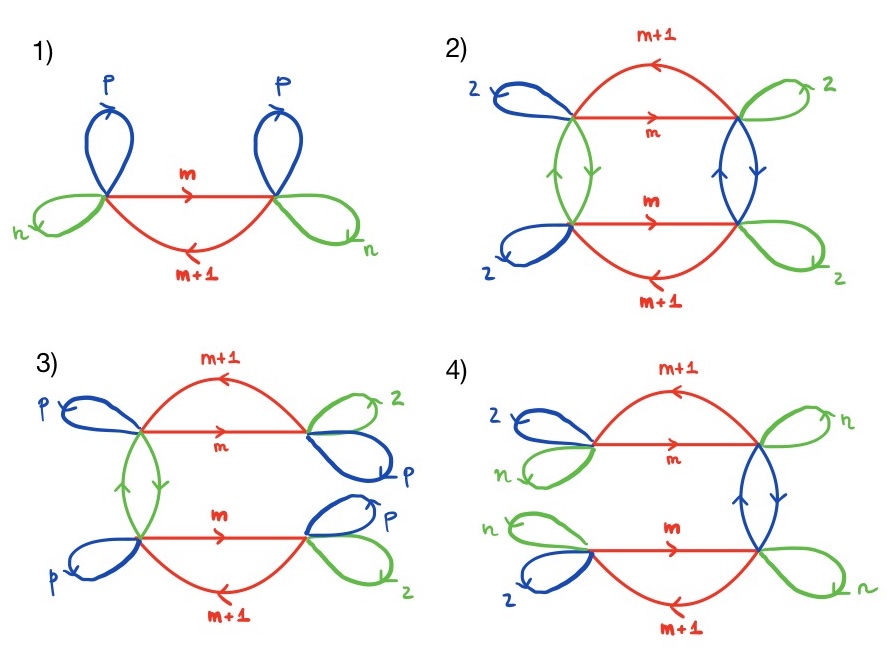}\caption{A non-trivial component of $F$, when $M$ is odd and $N,P$ are even.}\label{fig:fiber product one odd}\end{figure}
\end{lem}

\begin{proof}
As before, let $\{r_0,\dots, r_{2m}\}$, $\{g_0,\dots, g_{n-1}\}$, $\{g_n,\dots, g_{2n-1}\}$, 
$\{b_0,\dots, b_{p-1}\}$ and $\{b_{p}\,\dots, b_{2p-1}\}$ 
be the sets of cyclically ordered vertices of monochrome cycles such that
such that $v_r:=r_0 = g_0 = b_0$, $v_g:=r_{m} = g_{n} = b_1$ and $v_b:=r_{m+1} = g_{n-1} = b_{p}$. 
If $n,p>2$, then there is a connected components of $F$ containing a branching vertices $(v_b, v_r)$ and $(v_r, v_g)$, and distinct connected component containing the vertices $(v_r, v_b)$ and $(v_g, v_r)$. Each is a copy of the graph illustrated in Figure~\ref{fig:fiber product one odd}(1). 

If $n=p=2$, then there is one connected component of $F$ containing all fours branching vertices. It is illustrated in Figure~\ref{fig:fiber product one odd}(2). The cases where exactly one of $n,p$ is equal $2$ are illustrated in  Figure~\ref{fig:fiber product one odd}(3) and \ref{fig:fiber product one odd}(4).

All other components are simple monochrome cycles, or single vertices.
\end{proof}

\begin{lem}[All even]\label{lem:fiber all even} 
If $(M,N,P) = (2m,2n,2p)$ where $m,n,p\geq 3$, then the intersection $C^g\cap C$ for $g\in A$ is either trivial, 
or its conjugacy class is represented by a subgraph of the graph in Figure~\ref{fig:fiber product all even}.
\begin{figure}[h]\includegraphics[scale=0.3]{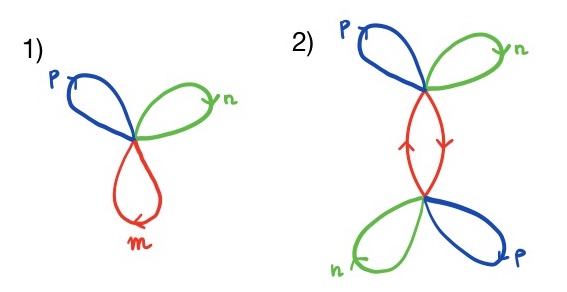}\caption{A non-trivial component of $F$, when $M,N,P$ are all even.}\label{fig:fiber product all even}\end{figure}
\end{lem}
\begin{proof}
As before, let $\{r_0,\dots, r_{m-1}\}$, $\{r_m,\dots, r_{2m-1}\}$, $\{g_0,\dots, g_{n-1}\}$, 
$\{g_n,\dots, g_{2n-1}\}$, $\{b_0,\dots, b_{p-1}\}$ and $\{b_{p}\,\dots, b_{2p-1}\}$ 
be the sets of cyclically ordered vertices of monochrome cycles such that
such that $v_r:=r_0 = g_0 = b_0$, $v_g:=r_{m} = g_{n} = b_1$ and $v_b:=r_{m+1} = g_{n-1} = b_{p}$. 
If $m,n,p>2$, then each connected components of $F$ contains at most one branching vertex. Any such connected component is a copy of the graph illustrated in Figure~\ref{fig:fiber product all even}(1). 
Otherwise, a connected component of $F$ contains at most two branching vertices. If $m = 2$, then the connected component containing $(v_g, v_b)$ also contains $(v_b, v_g)$ but no other branching vertices. Such connected component is illustrated in Figure~\ref{fig:fiber product all even}{2}. There are analogous graphs for for $n=2$, $v_r, v_b$ and for $p=2$, $v_r, v_g$. 
All other components are simple monochrome cycles, or single vertices.
\end{proof}

\begin{remark}\label{rem:all monochrome}If at least one of $M,N,P$ is even and $(M,N,P)\neq (2m+1, 4,4)$ (for any permutation), 
then all the simple cycles in the fiber product of $F$ are monochrome. It follows immediately from Lemmas~\ref{lem:fiber one even}, \ref{lem:fiber two even}, \ref{lem:fiber all even}.
\end{remark}

%%%%%%%%%%%%%%%%%%%%%%%%%%
\subsection{At least one even exponent}\label{sec:at least one even}
%%%%%%%%%%%%%%%%%%%%%%%%%%
We now will apply Theorem~\ref{thm:conditions for rf} to the twisted double $D(A,C,\beta)$ 
that is an index two subgroup of $\Art_{MNP}$ in Corollary~\ref{cor:splitting of 3 gen}. 
In this section we consider the case where at least one of $M,N,P$ is even.
Let $\mathcal A_{\rho}$ be the oppressive set of $C$ in $A$ with respect to $\rho:\overline X_{\sfrac 14}\to X_0$.

\begin{prop}\label{prop:triangle group quotient even}
Suppose $M,N,P\geq 4$ and at least one of $M,N,P$ is even. Suppose that $(M,N,P)\neq (2m+1, 4,4)$ (for any permutation). There exists a quotient $\phi:A\to \bar A$ such that 
\begin{enumerate}
\item $\bar A$ is virtually free,
\item $\bar C=\phi(C)$ is free and is malnormal in $\bar A$,
\item $\phi$ separates $C$ from $\mathcal A_{\rho}$,
\item $\beta:C\to C$ projects to an automorphism $\bar\beta:\bar C\to\bar C$. 
\end{enumerate}
\end{prop}

\begin{proof}
For each number $k$ define 
\[
\bar k = \left\{\begin{array}{cc}
\frac k 2& \text{if }k\text{ is even,}\\
k & \text{if } k\text{ is odd.}\end{array}\right.
\]
Let 
\[\bar A = \langle x,y,z \mid x^{\bar M}, y^{\bar N}, z^{\bar P}\rangle = \Z/\bar M\Z*\Z/\bar N\Z*\Z/\bar P\Z,\]
and let $\phi:A\to\bar A$ be the natural quotient. 
As a free product of finite groups $\bar A$ is virtually free.
Geometrically, we obtain $\bar A$ as the fundamental group of a $2$-complex $X_{\bullet}$ 
obtained from the bouquet of circles $X_0$ by attaching $2$-cells along $x^{\bar M}, y^{\bar N}$ and $z^{\bar P}$. 
Let $Y_{\bullet}$ be a $2$-complex obtained from $\overline X_{\sfrac1 4}$ by attaching a $2$-cell along each of the simple monochrome cycles with labels $x^{\bar M}, y^{\bar N}$ and $z^{\bar P}$. 
%Indeed, every cycle of $Y_{\bullet}$ that corresponds to a trivial element of $\bar C$ is mapped to a concatenation of monochrome cycles in $X_{\bullet}$.
The complex $Y_{\bullet}$ has the homotopy type of a graph (see Figure~\ref{fig:mapCtoA}), so $\pi_1 Y_{\bullet}$ is a free group.
There is an induced map $\rho_{\bullet}:Y_{\bullet}\to X_{\bullet}$ which lifts to an embedding $\widetilde Y_{\bullet}\to \widetilde X_{\bullet}$ of the universal covers. 
We have $\pi_1 Y_{\bullet} = \bar C$.
By Lemma~\ref{lem:embedded universal covers}, $\phi$ separates $C$ from $\mathcal A_{\rho}$.
%Since $$\chi(Y ) = 3-9+6-\#\{\text{odd numbers among }M,N,P\}$$ (see Figure~\ref{fig:mapCtoA}),
%the group $\bar C$ is a free group with $\rank \bar C = 1+\#\{\text{odd numbers among }M,N,P\}$.
%

In Lemma~\ref{lem:fiber one even}, Lemma~\ref{lem:fiber two even}  and Lemma~\ref{lem:fiber all even}, 
we computed the graphs representing the intersections $C\cap C^g$ for $g\in A$. 
The intersections of $\bar C\cap \bar C^{\bar g}$ for $\bar g\in \bar A$ can be represented by the graphs obtained in those lemmas with $2$-cells added along simple monochrome cycles with labels $x^{\bar M}, y^{\bar N}$ and $z^{\bar P}$. 
The graphs become contractible after attaching $2$-cells to the simple monochrome cycles (see Remark~\ref{rem:all monochrome}). 
It follows that $\bar C$ is malnormal in $\bar A$.

The $2$-cells of $Y_{\bullet}$ can be pulled back along the homotopy equivalence $ X_{\sfrac1 4}\to\overline X_{\sfrac1 4}$. 
See Figure~\ref{fig:2cellspulledback}. 
\begin{figure}[h]\includegraphics[scale=0.25]{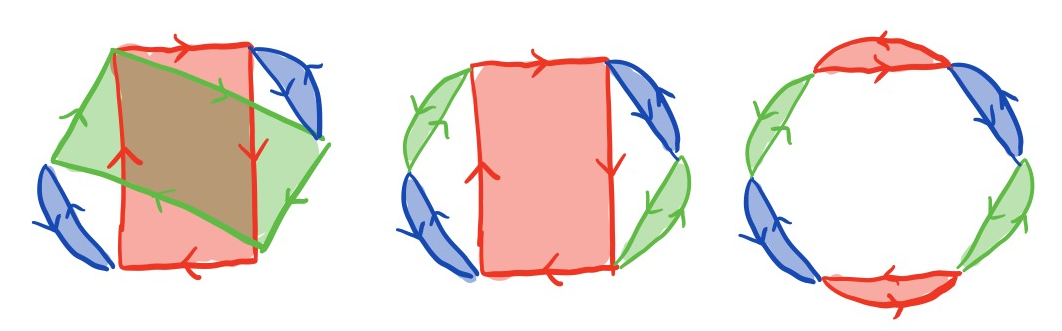}\caption{The $2$-cells in the presentation complex of $\bar A$ can be pulled back to $X_{\sfrac 14}$. These are three cases where at least one of $M,N,P$ is even. These new $2$-complexes admit a rotation by $\pi$ which represents the automorphism $\beta$.}\label{fig:2cellspulledback}\end{figure}
The pulled back $2$-cells in Figure~\ref{fig:2cellspulledback} have boundary cycles 
that are denoted by the same colors as the corresponding boundary cycles 
of the corresponding $2$-cells in $\overline X_{\sfrac 14}$.
By Observation~\ref{obs:beta bar}, $\beta$ projects to an automorphism $\bar\beta:\bar C\to \bar C$. 
\end{proof}

Since free groups are locally quasiconvex, $\bar C$ is quasiconvex in $\bar A$. By combining Proposition~\ref{prop:triangle group quotient even} with Theorem~\ref{thm:conditions for rf} we have the following.
\begin{cor}\label{cor:at least one even}
If at least one $M,N,P$ is even and $(M,N,P)\neq (2m+1, 4,4)$ (for any permutation), then $\Art_{MNP}$ splits as an algebraically clean graph of finite rank free groups. In particular, $\Art_{MNP}$ is residually finite.
\end{cor}

%%%%%%%%%%%%%%%%%%%%%%%%%%
\subsection{All exponents odd}\label{sec:all exponents odd}
%%%%%%%%%%%%%%%%%%%%%%%%%%
We will now apply Theorem~\ref{thm:conditions for rf} in the case where $M,N,P$ are all odd.
Again, let $\mathcal A_{\rho}$ be the oppressive set of $C$ in $A$ with respect to ${\rho}:\overline X_{\sfrac 14}\to X_0$.
The main goal of this section is the following.
\begin{prop}\label{prop:triangle group quotient}
Suppose $(M,N,P) = (2m+1, 2n+1, 2p+1)$ where $m,n,p\geq 2$. There exists a quotient $\phi:A\to \bar A$ such that 
\begin{enumerate}
\item $\bar A$ is a hyperbolic von Dyck group,
\item $\bar C:=\phi(C)$ is a free group of rank $2$ and is malnormal in $\bar A$,
\item $\phi$ separates $C$ from $\mathcal A_{\rho}$,
\item $\beta:C\to C$ projects to an automorphism $\bar\beta:\bar C\to\bar C$.
\end{enumerate}
\end{prop}

\begin{proof}
Let $\phi:A\to\bar A$ be the natural quotient 
where $\bar A$ is given by the presentation
\begin{equation}\label{eq:presentation}
\tag{$*$}
\bar A = \langle x,y,z\mid x^M, y^N, z^P, x^{m}y^{n}z^{p}\rangle.
\end{equation}
The group $\bar A$ is the \emph{von Dyck group} $D(M,N,P)$. 
Remind, $D(M,N,P)$ is the index two subgroup of the group of reflection of a triangle in $\mathbb H^2$ with angles $\frac{\pi}{M},\frac{\pi}{N},\frac{\pi}{P}$, 
and can be given by the presentation
\begin{equation}\label{eq:presentation2}
\tag{$**$}
D(M,N,P) = \langle a,b,c\mid a^M, b^N, c^P, abc\rangle.
\end{equation}
In order to see that $\bar A$ is isomorphic to $D(M,N,P)$, 
note that $x^m, y^n, z^p$ are generators of $\bar A$. 
Indeed, since $m(M-2) = m(2m-1) = M(m-1)+1$, we have 
\[(x^m)^{M-2} = x^{M(m-1)+1} = x \] 
and similarly $(y^n)^{N-2}=y$ and $(z^p)^{P-2}=z$. 
By setting $a = x^m$, $b=y^n$ and $c= z^p$, and rewriting the presentation in generators $a,b,c$,
we get the presentation (\ref{eq:presentation2}).

Let $X_{\bullet}$ be the presentation complex of \eqref{eq:presentation}. 
The $1$-skeleton of $X_{\bullet}$ can be identified with $X_0$.
Let $Y_{\bullet}$ be a $2$-complex obtained from $\overline X_{\sfrac 14}$ by attaching the following $2$-cells
\begin{itemize}
\item one simple monochrome cycle with label $x^M,y^N, z^P$ respectively for each color,
\item two copies of a $2$-cell with the boundary word $x^{m}y^{n}z^{p}$. 
\end{itemize}
By Lemma~\ref{lem:phi(C) free} %or Lemma~\ref{lem:phi(C) free 3} 
(stated after this proof), $\pi_1Y_{\bullet} = \bar C$. 
In Lemma~\ref{lem:fiber all odd}, we computed the graph representing an intersection $C\cap C^g$ for $g\in A$. 
The intersection $\bar C\cap \bar C^{\bar g}$ for $\bar g\in \bar A$ can be represented by a $2$-complex obtained from that graph by attaching the $2$-cells as along all cycles with labels $x^M, y^N, z^P, x^{m}y^{n}z^{p}$. 
After attaching the $2$-cells the complex becomes contractible. 
Thus $\bar C$ is malnormal in $\bar A$.

We now show that $\beta$ projects to $\bar C$.
As in proof of Proposition~\ref{prop:triangle group quotient even}, all the $2$-cells of $Y_{\bullet}$ can be pulled back along the homotopy equivalence $X_{\sfrac1 4}\to\overline X_{\sfrac1 4}$.
See Figure~\ref{fig:2cellspulledbackodd} for the five $2$-cells that we attach to $X_{\sfrac1 4}$ and that correspond to the five $2$-cells of $Y_{\bullet}$. 
\begin{figure}[h]\includegraphics[scale=0.25]{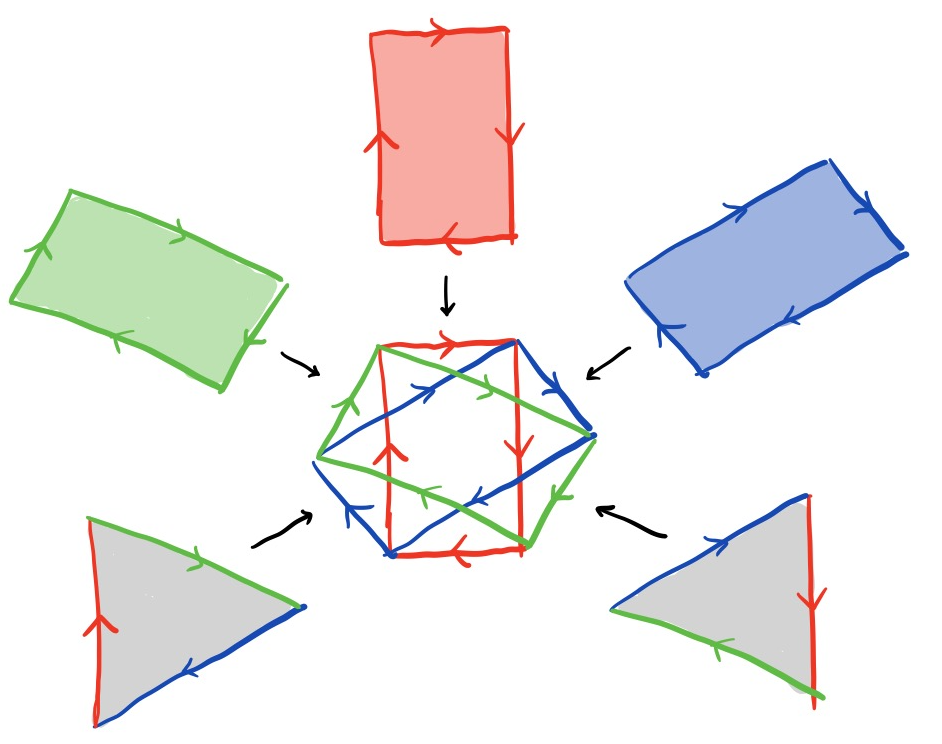}\caption{The $2$-cells in the presentation complex of $\bar A$ can be pulled back to $X_{\sfrac 14}$ in the case where $M,N,P$ are all odd. 
The new $2$-complex admits the $\pi$-rotation which represents the automorphism $\beta$. The rotation exchanges the two triangular $2$-cells and leaves other $2$-cells invariant.}\label{fig:2cellspulledbackodd}\end{figure}
Three of the $2$-cells pulled back to $X_{\sfrac 14}$ in the figure have boundary cycles 
that are denoted by the same colors as the corresponding boundary cycles 
of the corresponding $2$-cells in $\overline X_{\sfrac 14}$.
The remaining two have boundary cycles of length three and correspond to the two copies of a $2$-cell with the boundary $x^{m}y^{n}z^{p}$ in $\overline X_{\sfrac 14}$.
By Observation~\ref{obs:beta bar}, $\beta$ projects to an automorphism $\bar\beta:\bar C\to \bar C$. 

Finally, it remains to prove that $\phi$ separates $C$ from $\mathcal A$. Let $X'_{\bullet}$ be the presentation complex of (\ref{eq:presentation2}), and let $Y'_{\bullet}$ be a $2$-complex with the $1$-skeleton as in Figure~\ref{fig:von Dyck subgroup}, three monochrome $2$-cells and two with boundary word $abc$. 
\begin{figure}[h]\includegraphics[scale=0.25]{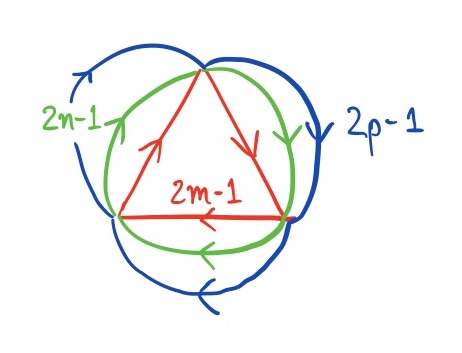}\caption{The graph $Y'$. Red arrows correspond to generator $a$, green to $b$, and blue to $c$.}\label{fig:von Dyck subgroup}\end{figure}
There is a natural immersion $\rho'_{\bullet}:Y'_{\bullet}\to X'_{\bullet}$ inducing the inclusion $\bar C\to \bar A$. 
Let $Y'$ and $X'$ be the $1$-skeleta of $Y'_{\bullet}$ and $X'_{\bullet}$ respectively, and let $\rho':Y'\to {X'}$ be the map $\rho'_{\bullet}$ restricted to the $1$-skeleta.
%The map of the $1$-skeleta $\rho':Y'\to {X'}$ where $Y'$ and $X'$ are the $1$-skeleta of $Y'_{\bullet}$ and $X'_{\bullet}$ respectively. induces 
The map $\rho'$ is an inclusion of $\pi_1 Y' \simeq F_7$ in $\pi_1 {X'}\simeq F_3$.
In terms of the original generators of $A$, we have $\pi_1 {X'} = \langle x^m, y^n, z^p\rangle$, so  this is a different inclusion $F_7\to F_3$ than $C\to A$. 
However, the image $\phi'(\mathcal A_{\rho}')\subseteq \bar A$ of the oppressive set $\mathcal A_{\rho'}$ 
with respect to ${\rho'}$ is equal to $\phi(\mathcal A_{\rho})\subseteq \bar A$. 
Indeed, all the pairs of paths $\mu_1,\mu_2$ in $Y'$ that $\rho'(\mu_1)\cdot\rho'(\mu_2)$ is a closed path are in one-to-one correspondence with such pairs of paths in $\overline X_{\sfrac 14}$ (see Figure~\ref{fig:von Dyck subgroup} for $Y'$ and Figure~\ref{fig:mapCtoA} for $\overline X_{\sfrac 14}$). 
Thus to show that $\phi$ separates $C$ from $\mathcal A_{\rho}$, 
it suffices to show $\phi'(\mathcal A_{\rho'})$ is disjoint from $\bar C$ in $\bar A$.

Let $\widetilde X'_{\bullet}$ denote the universal cover of $X'_{\bullet}$ with the $2$-cells with the same boundary identified
(i.e.\ $M$ copies of the $2$-cell whose boundary word is $a^M$ are collapsed to a single $2$-cell, and similarly with $b^N, c^P$). 
The complex $\widetilde X'_{\bullet}$ admits a metric so that that makes it isometric to $\mathbb H^2$. 
In particular, $\widetilde X'_{\bullet}$ is CAT(0). Consider the induced metric on $Y'_{\bullet}$. 
Since $\rho'_{\bullet}$ is an immersion, a lift $\widetilde Y'_\bullet \to\widetilde X'_{\bullet}$ is a local isometric embedding 
(i.e.\ every point in $\widetilde Y'_{\bullet}$ has a neighborhood 
such that the restriction of $\widetilde Y'_{\bullet}\to\widetilde X_{\bullet}'$ to that neighborhood is an isometry onto its image), 
and by \cite[Proposition II.4.14]{BridsonHaefliger}, it is an embedding. 
By Lemma~\ref{lem:embedded universal covers}, $\phi'$ separates $\pi_1Y'$ from $\mathcal A_{\rho'}$. 
This means that $\bar C$ is disjoint from $\phi'(\mathcal A_{\rho'}) = \phi(\mathcal A_{\rho})$, 
and so $\phi$ separates $C$ from $\mathcal A_{\rho}$.
\end{proof}

We will prove the last missing bit in Lemma~\ref{lem:phi(C) free}. First, we recall a version of the ping-pong lemma and its application in the hyperbolic plane, which allows us to show that certain convex subsets of $\mathbb H^2$ are disjoint.
\begin{lem}[Ping-pong Lemma]\label{lem:ping-pong}
Let a group generated by $u,v$ act on a set $\Omega$ 
and let $U_+, U_-,V_+, V_-$ be disjoint subsets of $\Omega$ 
such that
\[u(\Omega - U_-)= U_+,\]
\[v(\Omega - V_-) = V_+,\]
Then $u,v$ freely generate a free group.
\end{lem}

%In Section~\ref{sec:all exponents odd} we apply the ping-pong lemma 
%to certain subspaces of the hyperbolic plane $\mathbb H^2$, 
%which are obtained as intersection and union of halfplanes in $\mathbb H^2$. 
%The following lemma allows us to show that these subspaces are disjoint.

\begin{lem}\label{lem:hyperbolic lines}
Let $ABCD$ be a convex quadrangle in $\mathbb H^2$ with all internal angles $\leq \frac{\pi}2$. 
Then the lines $\overline{AB}$ and $\overline{CD}$ do not intersect 
in $\mathbb H^2\cup \partial \mathbb H^2$. \end{lem}

\begin{proof}
Two lines in $\mathbb H^2$ do not intersect in $\mathbb H^2\cup \partial \mathbb H^2$ 
if and only if there exists a common perpendicular line, 
i.e.\ a line that intersects each of the two lines at angle $\frac \pi 2$.
Consider the shortest geodesic segment $p$ between segments $AB$ and $CD$. 
The segment $p$ is contained inside the closed quadrangle $ABCD$, by the assumption on the angles of $ABCD$. 
Moreover, the angles between $p$ and each of the segments $AB$, $CD$ are equal $\frac{\pi}{2}$.
This proves that the line containing $p$ is perpendicular to the lines $\overline{AB}$ and $\overline{CD}$.
%
%Let $A^{\infty}, B^{\infty}$ be the endpoints of $\overline{AB}$ in $\partial \mathbb H^2$ 
%such that $A^{\infty}, A,B, B^{\infty}$ lie on $\overline{AB}$ in the given order. 
%Similarly let $C^{\infty}, D^{\infty}$ be the endpoints of $\overline{CD}$ in $\partial \mathbb H^2$ 
%such that $C^{\infty}, C, D,D^{\infty}$ lie on $\overline{CD}$ in the given order. 
%By the assumption, $\angle DAB^\infty\leq \frac{\pi}2\leq \angle CBB^{\infty}\leq \angle DBB^\infty$. 
%By continuity, there exists a point $A'$ in the segment $AB$ 
%such that $\angle DA'B^{\infty} = \angle DA'A^{\infty} = \frac \pi 2$. 
%Similarly, there exists a point $B'$ in the segment $AB$ 
%such that $\angle CB'B^{\infty} = \angle CB'A^{\infty} = \frac \pi 2$. 
%Note that the point $B'$ must lie in the segment $A'B$ and not in $AA'$, 
%because otherwise $A', B'$ and the point of the intersection of $A'D$ and $B'C$ 
%would be a triangle with sum of the angles $>\pi$, which is impossible. 
%Consider all the lines perpendicular to $\overline{AB}$. 
%This includes the line containing segments $A'D$ and the line containing the segment $B'C$.
%By the assumption, $\angle A'DC^{\infty}\leq \frac \pi 2\leq B'CC^{\infty}$, 
%by continuity there exists a line $p$ in that collection of lines perpendicular to $\overline{AB}$
% that is also perpendicular to $\overline{CD}$.
\end{proof}

We are now ready to complete the proof of Proposition~\ref{prop:triangle group quotient}. The group $\bar C$ and the complexes $Y_{\bullet}, X_\bullet$ are as in the proof of Proposition~\ref{prop:triangle group quotient}.

%We first assume that $M,N,P\geq 5$.
\begin{lem}\label{lem:phi(C) free} Let $M,N,P\geq 5$. The group $\bar C$ in the proof of Proposition~\ref{prop:triangle group quotient} is the fundamental group of the $2$-complex $Y_{\bullet}$ %obtained from $\overline X_{\sfrac 14}$ by attaching the following $2$-cells
%\begin{itemize}
%\item one monochrome cycle of length $M,N$ or $P$ respectively for each color,
%\item two copies of a $2$-cell with the boundary $x^{m}y^{n}z^{p}$,
%\end{itemize}
and the map $Y_{\bullet}\to X_{\bullet}$ induces the inclusion of group $\bar C\to \bar A$.
In particular, $\bar C$ is a free group of rank $2$.
\end{lem}
\begin{proof} 
It is clear that the $2$-cells in $X_{\bullet}$ pull back to the five $2$-cells of $Y_{\bullet}$, 
so $\bar C$ is necessarily the image of $\pi_1Y_{\bullet} $ in $\bar A$. 
By pushing free edges into the $2$-cells, we can show that the wedge based at the $a_+/b_-$ 
(the top vertex in $\overline X_{\sfrac 14}$ in Figure~\ref{fig:mapCtoA}) of two loops with boundary words $x^m y^{-n}$ and $z^{-p} x^{m}$ is a retract of $Y$. %The complex $Y$ is hom
%eomorphic to a surface with boundary, and since $\chi(Y) = 3-9+5 = -1$, we get 
In particular, $\pi_1 Y= F_2$.
In order to show that $\pi_1Y_{\bullet} = \bar C$,
we will show that $\pi_1Y_{\bullet}$ maps to a free group of rank two in $\bar A = \pi_1X_{\bullet}$. 
%
%Let the vertex $a_+/b_-$ 
%(the top vertex in $\overline X_{\sfrac 14}$ in Figure~\ref{fig:mapCtoA}) 
%be the basepoint. 
We will show that the elements $u=x^m y^{-n}$ and $v=z^{-p} x^{m}$ generate $F_2$ in $\bar A$.
In the generators $a,b,c$ of $\bar A$ as in presentation~\eqref{eq:presentation2} given above, 
we have $u = ab^{-1}$ and $v=c^{-1}a$. 
The group $\bar A$ is an index two subgroup of a reflection group 
generated by the reflection in the sides of triangle 
with angles $\frac{\pi}{2m+1}, \frac{\pi}{2n+1},\frac{\pi}{2p+1}\leq \frac {\pi}5$ in $\mathbb H^2$. 
Therefore $\bar A$ preserves the tilling of $\mathbb H^2$ with triangles with those angles. 
See Figure~\ref{fig:tiling}. 
\begin{figure}\includegraphics[scale=0.2]{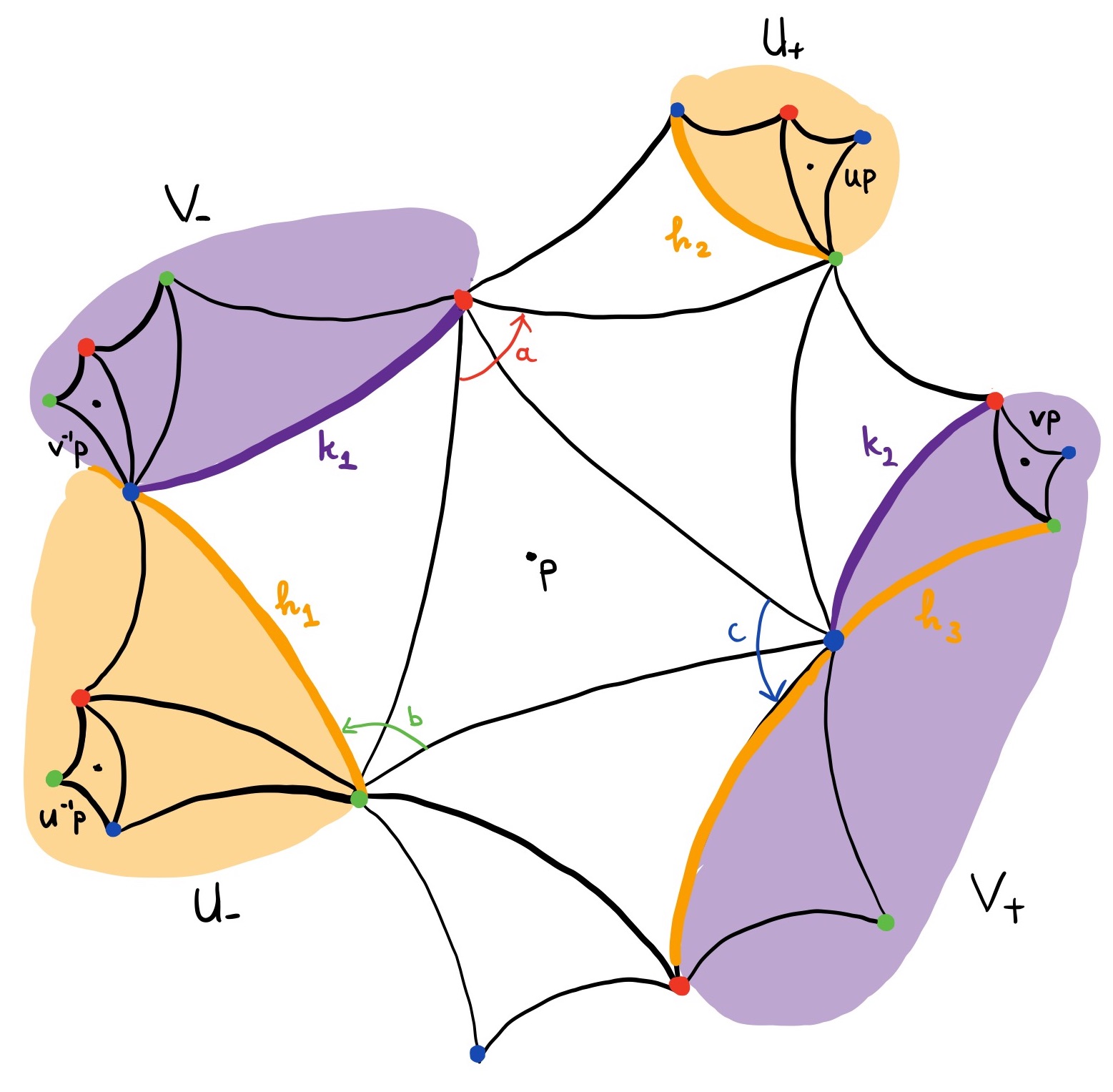}
\caption{A portion of the hyperbolic plane tilling with a triangle whose all three angles are $\frac {\pi}5$.}\label{fig:tiling}\end{figure}
We use the hyperplanes from this tiling to define subsets $U_+, U_-,V_+, V_-$ 
and apply Lemma~\ref{lem:ping-pong}. 
Let $P_a, P_b, P_c$ be three vertices of a triangle in the tiling
such that the isometry $a$ fixes $P_a$, $b$ fixes $P_b$ and $c$ fixes $P_c$.
Let $k_1$ be the line $a^{-1}(\overline{P_aP_b})$, and let $h_1$ be the line $b(\overline{P_bP_c})$. 
The lines $k_1$ and $h_1$ intersect, see Figure~\ref{fig:tiling}.
Let 
\[h_2 := uh_1 ,\]
\[k_2 := vk_1,\]
\[ h_3 := vh_1.\]
Clearly $k_2$ and $h_3$ intersect. 
We claim that no other pairs of lines among $h_1, h_2, h_3, k_1, k_2$ intersect.
Since $M,N,P\geq 5$, all angles in all triangles are $\leq \frac \pi 5$. 
For each pair of hyperplanes that we claim are disjoint, 
there exists a geodesic quadrangle with two opposite sides lying in those hyperplanes, 
and with all angles $\leq \frac {\pi}2$. 
By Lemma~\ref{lem:hyperbolic lines} such hyperplanes are disjoint. 

Let $U_+$ be the closed outward halfplane of $h_2$, 
i.e.\ the halfplane that does not contain any of $h_1, h_3, k_1, k_2$. 
Let $U_-$ be the open outward halfplane of $h_1$. 
We clearly have $u(\mathbb H^2-U_-) = U_+$. 
Now, let $V_+$ be the union of the closed outward halfplanes of $k_2$ and $h_3$ 
(i.e.\ the halfplanes not containing $h_1, k_1$ or $h_2$), 
and let $V_-$ be the intersection of the open outward halfplanes of $k_1$ 
and the open inward halfplane of $h_1$. 
We have $v(\mathbb H^2-V_-) = V_+$. 
The subspaces $U_+, U_-,V_+, V_-$ are pairwise disjoint. 
By Lemma~\ref{lem:ping-pong}, $u$ and $v$ freely generate a free group.
\end{proof}

%\begin{lem}\label{lem:phi(C) free 3} Lemma~\ref{lem:phi(C) free} also holds when $M=3$ and $N,P\geq 5$.
%\end{lem}
%\begin{proof}
%Consider the hyperplanes $h_1,h_2, h_3, k_1,k_2$ as in the proof of Lemma~\ref{lem:phi(C) free}. 
%In this case $k_1$ and $h_2$ also intersect. 
%There are no other new intersections between $h_1,h_2, h_3, k_1,k_2$. 
%Let $h_4 := vh_2$.
%See Figure~\ref{fig:tiling355}. 
%\begin{figure}\includegraphics[scale=0.2]{tiling355}
%\caption{Tiling by triangles with angles $\frac\pi 3, \frac \pi 5, \frac \pi 5$.}\label{fig:tiling355}\end{figure}
%Similarly as before, we argue that $h_4$ only intersects $k_2$ among $h_1,h_2, h_3, k_1,k_2$.
%We set $U_+$ and $U_-$ as before. 
%We define $V_+$ as the union of the closed outward halfplanes of $k_2$, $h_3$ and $h_4$,
%and $V_-$ is the intersection of the open halfplanes: outward of $k_1$ and inward of $h_1$ and $h_2$. 
%As before we have $u(\mathbb H^2-U_-) = U_+$ and $v(\mathbb H^2-V_-) = V_+$, 
%and $U_+, U_-, V_+, V_-$ are pairwise disjoint.
%By Lemma~\ref{lem:ping-pong}, $u$ and $v$ freely generate a free group.
%\end{proof}

By Proposition~\ref{prop:triangle group quotient} and Theorem~\ref{thm:conditions for rf} we get the following.
\begin{cor}\label{cor:all odd}
If $M,N,P\geq 5$ are all odd, then the group $\Art_{MNP}$ splits as an algebraically clean graph of finite rank free groups. In particular, $\Art_{MNP}$ is residually finite.
\end{cor}

%The above lemma does not hold for $(M,N,P)=(3,3,2p+1)$.
%Indeed, in that cases $\bar C$ is a proper quotient of $\pi_1 Y$ described above.

We finish this section with the analogous construction as in the proof of Proposition~\ref{prop:triangle group quotient} but in the case of Artin group $\Art_{333}$.
\begin{exa} 
If $(M,N,P) = (3,3,3)$ then group $C$ has index $3$ in $A$. 
%also $\bar C$ has finite index in $\bar A$. 
Let $\bar A$ is the Euclidean von Dyck group $D(3,3,3)$ obtained in the same way as in proof of Proposition~\ref{prop:triangle group quotient}.
Then the subgroup $\bar C = \mathbb Z^2$. 
Indeed, the complex $Y_{\bullet}$ has one additional $2$-cell 
whose boundary reads the third copy of the word $xyz$. 
This complex is homeomorphic to a closed surface with $\euler (Y_{\bullet}) = 3-9+6 =0$, 
so $Y_{\bullet}$ is homeomorphic to a torus.
Note that the $2$-cells of $Y_{\bullet}$ still can be pulled back to $X_{\sfrac 14}$. 
The third triangle pulls back to a hexagon, which is invariant under the graph automorphism $b$. 
Thus it is still true that $\beta$ projects to $\bar C$. 
\end{exa}

%\begin{exa} Let $M=N=3$, and $P=2p+1\geq 5$. 
%The subgroup of $\bar A = \langle a,b,c\mid a^3, b^3, c^P, abc\rangle$ 
%generated by $u = ab^{-1}$ and $v = c^{-1}a$ is not free. 
%Indeed, 
%\[
%[v,u^{-1}] = c^{-1}aba^{-1}a^{-1}cab^{-1} = c^{P-1}abacab^2 =  c^{P-2}acab^2 = c^{P-2}ab = c^{P-3}, 
%\]
%so in particular $[u,v]^P = 1$. This also shows that the quotient $\phi:A\to \bar A$ does not separate $C$ from $\mathcal A$, because $c^{P-3}\in \mathcal A$.
%\end{exa}

%%%%%%%%%%%%%%%%%%%%%%%%%%%%%%%%%%%%%%%%%%%%%%%%%%%%
\section{Residual finiteness of more general Artin groups}\label{sec:rf of other}
%%%%%%%%%%%%%%%%%%%%%%%%%%%%%%%%%%%%%%%%%%%%%%%%%%%%
The proof of residual finiteness of a three generator Artin group where at least one exponent is even, generalizes to other Artin groups.
Throughout this section $\Gamma$ is a graph admitting an admissible partial orientation,
so by Theorem~\ref{thm:decomposition} $\Art_\Gamma$ splits as 
a free product with amalgamation or an HNN extension of finite rank free groups.

\begin{thm}\label{thm:generalization at least one even}
If all the simple cycles in nontrivial connected components of $F$ are monochrome, then $\Art_{\Gamma}$ is residually finite.
\end{thm} 
\begin{proof}
This proof is analogous to the proof of Proposition~\ref{prop:triangle group quotient even}. 
The quotient $\bar A$ of $A$ is obtained by adding a relation $x^{\bar M}$ for each generator $x$ of $A$ corresponding to an edge in $\Gamma$ with label $M$ and where $\bar M$ is either $\frac M2$ or $M$, depending on parity of $M$. 
Then $\bar A$ is virtually free, and $\bar C$ is free. 
The assumption that simple cycles in nontrivial connected components of $F$ are monochrome, ensures that $\bar C$ is malnormal. 
The universal cover $\widetilde X_{\bullet}$ of the Cayley $2$-complex of $\bar A$ can be homotoped to a tree by replacing each monochrome $2$-cycle corresponding to a $x^{\bar M}$ with an $\bar M$-star graph whose middle vertex corresponds to the $2$-cell and other vertices correspond to the original vertices. 
We note that the presentation $2$-complex $Y_{\bullet}$ of $\bar C$ can also be homotoped to a graph in that way. 
It follows that the map $\widetilde Y_{\bullet}\to \widetilde X_{\bullet}$ is a local isometric embedding, and consequently an embedding, by~\cite[Proposition II.4.14]{BridsonHaefliger}.
By Lemma~\ref{lem:embedded universal covers} $\phi$ separates $C$ from the oppressive set $\mathcal A$ of $C$ in $A$.
\begin{com}referees comment\end{com}
All the attached $2$-cells of $\overline X_{\sfrac 14}$ can be pulled back to $X_{\sfrac 14}$ in a way that $\beta$ projects to $\bar \beta$.
Depending on whether $\overline X_{\sfrac 14}$ is connected or not, the conclusion follows from Theorem~\ref{thm:conditions for rf} or Theorem~\ref{thm:conditions for rf hnn}.
\end{proof}

\begin{cor}
Let $\Gamma$ be a graph admitting an admissible partial orientation. 
If all labels are even and $\geq 6$, then $\Art_{\Gamma}$ is residually finite. 
\end{cor}
\begin{proof} 
For every color with corresponding label $2m$, there are three segments of that color in $\overline X_{\sfrac 14}$, which have lengths $1,m-1,m$ respectively. 
The segments of the length $1$ and $m-1$ form one cycle and the other segment forms its own cycle.
Since the branching vertices in the fiber product $F$ are pairs of branching vertices, a lift of every monochrome cycle has exactly one branching vertex. 
It follows that all simple cycles in nontrivial connected components of $F$ are monochrome.
By Theorem~\ref{thm:generalization at least one even}, we are done.
\end{proof}

There are many more examples of graphs satisfying the assumption of Theorem~\ref{thm:generalization at least one even}.
However, in the following example, Theorem~\ref{thm:generalization at least one even} cannot be applied to any admissible partial orientation of $\Gamma$.

\begin{exa}Let $\Gamma$ be the graph on the left in Figure~\ref{exa:badfiberproduct}.
Note that every admissible partial orientation of $\Gamma$ is the same up to a permutation of the vertex labels.
The second picture in Figure~\ref{exa:badfiberproduct} is a part of the graph $X_{\sfrac 14}$.  Edges that are thickened get collapsed in $\overline X_{\sfrac 14}$, see the next graph.
Finally, on the right we have a cycle that admits two distinct combinatorial immersion to $\overline X_{\sfrac 14}$.
This yields a non monochrome simple cycle in $F$.
\begin{figure}\includegraphics[scale=0.25]{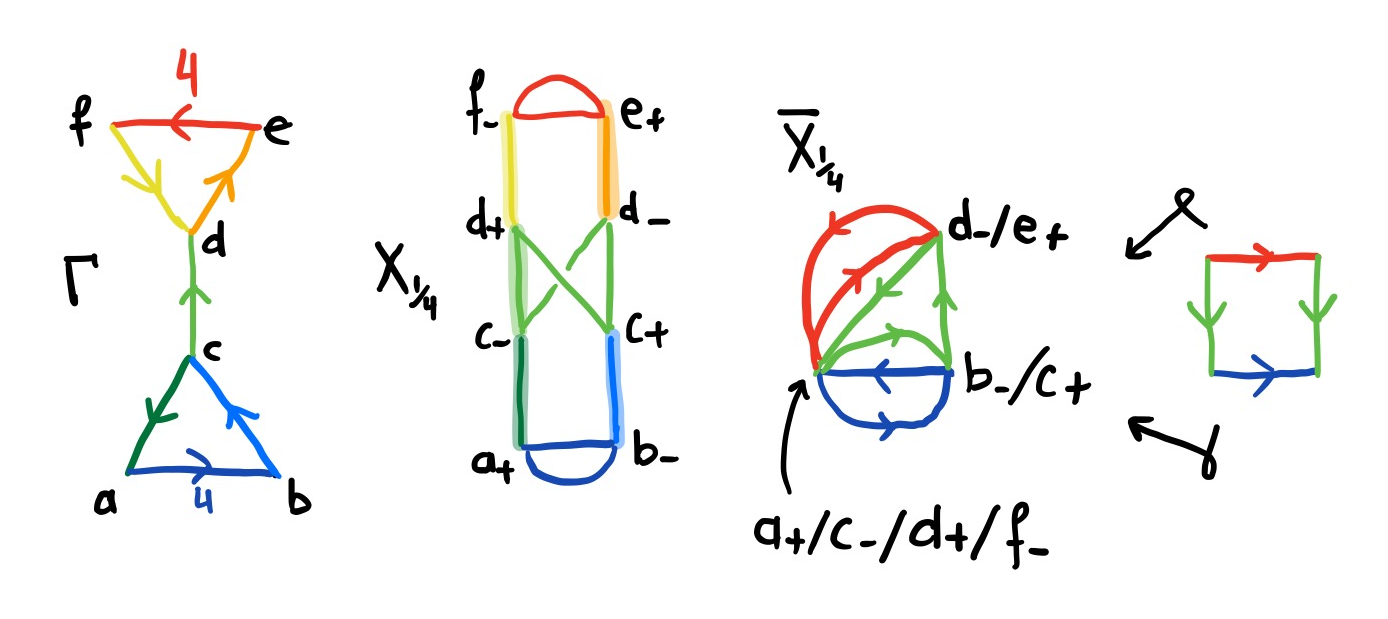}
\caption{In the above example $M_{ab} =M_{ef} = 4$ and $M_{cd}$ is odd. For each edge $e$ of $\Gamma$. The second graph is a part of $X_{\frac 14}$ and the third graph is the image of that part of $X_{\sfrac 14}$ in $\overline X_{\sfrac 14}$, which admits two different combinatorial immersions of the cycle on the right.}\label{exa:badfiberproduct}\end{figure}
\end{exa}

\bibliographystyle{alpha}
\bibliography{../../../kasia}

\end{document}